\documentclass[11pt]{amsart}

\usepackage{amsmath}
\usepackage{amssymb}
\usepackage{amsfonts}
\usepackage{amsthm}
\usepackage{enumerate}

\usepackage{graphicx}
\usepackage{verbatim}

\newcommand{\Be}{\begin{equation}}
\newcommand{\Ee}{\end{equation}}

\newcommand{\Bm}{\begin{multline}}
\newcommand{\Em}{\end{multline}}
\newcommand{\Bea}{\begin{eqnarray}}
\newcommand{\Eea}{\end{eqnarray}}
\newcommand{\Beas}{\begin{eqnarray*}}
\newcommand{\Eeas}{\end{eqnarray*}}
\newcommand{\Benu}{\begin{enumerate}}
\newcommand{\Eenu}{\end{enumerate}}
\newcommand{\Bi}{\begin{itemize}}
\newcommand{\Ei}{\end{itemize}}

\def\intslash{\rlap{\kern  .32em $\mspace {.5mu}\backslash$ }\int}
\def\qsl{{\rlap{\kern  .32em $\mspace {.5mu}\backslash$ }\int_{Q_x}}}

\def\vth{\vartheta}

\def\tp{{\tilde p}}
\def\tq{{\tilde q}}

\def\emph#1{{\it #1 }}

\def\ga{\gamma}

\def\cf{{\it cf}}

\def\loc{{\text{\rm loc}}}

\def\rad{{\text{\rm rad}}}
\def\sph{{\text{\rm sph}}}

\def\ev{{\text{\rm ev}}}

\def\inn#1#2{\langle#1,#2\rangle}

\def\noi{\noindent}

\def\lc{\lesssim}
\def\gc{\gtrsim}

\def\eps{\varepsilon}

\def\ka{\kappa}
            
\def\la{\lambda}

\def\vphi{\varphi}
             
\def\om{\omega}

\def\fA{{\mathfrak {A}}}
\def\fB{{\mathfrak {B}}}

\def\fM{{\mathfrak {M}}}

\def\bbR{{\mathbb {R}}}

\def\bbT{{\mathbb {T}}}

\def\bbZ{{\mathbb {Z}}}

\def\cB{{\mathcal {B}}}
\def\cC{{\mathcal {C}}}

\def\cF{{\mathcal {F}}}

\def\cH{{\mathcal {H}}}
\def\cI{{\mathcal {I}}}

\def\cK{{\mathcal {K}}}

\def\cM{{\mathcal {M}}}

\def\cR{{\mathcal {R}}}
\def\cS{{\mathcal {S}}}
\def\cT{{\mathcal {T}}}

 at 10 true pt

\def\be#1{\begin{equation}\label{ #1}}
\def\endeq{\end{equation}}
\def\endal{\end{align}}
\def\bas{\begin{align*}}
\def\eas{\end{align*}}
\def\bi{\begin{itemize}}
\def\ei{\end{itemize}}

\def\floc{{\rm LF}}

\def\eps{\varepsilon}

\def\emph#1{{\it #1}}
\def\textbf#1{{\bf #1}}

\def\bmod{\frak b}
\def\gmod{\frak g}

\theoremstyle{plain}
   \newtheorem{theorem}{Theorem}[section]

   \newtheorem{proposition}[theorem]{Proposition}
   
   \newtheorem{lemma}[theorem]{Lemma}
   \newtheorem{corollary}[theorem]{Corollary}
   \newtheorem{observation}[theorem]{Observation}

\theoremstyle{remark}

\theoremstyle{definition}

\numberwithin{equation}{section}

\begin{document}

\title{
Characterizations of
Hankel multipliers}

\author{Gustavo Garrig\'os and Andreas Seeger}

\address{G. Garrig\'os \\
Dep. Matem\'aticas C-XV \\
Universidad Aut\'onoma de Madrid\\
28049 Madrid, Spain}
\email{gustavo.garrigos@uam.es}

\address{A. Seeger   \\
Department of Mathematics\\ University of Wisconsin-Madison\\Madison, WI 53706, USA}
\email{seeger@math.wisc.edu}


\begin{thanks} {G.G.
partially supported by grant ``MTM2007-60952'' and \emph{Programa
Ram\'on y Cajal}, MCyT (Spain). A.S. partially supported
by NSF grant DMS 0652890.}
\end{thanks}


\begin{abstract} We give characterizations of radial Fourier multipliers as acting on radial $L^p$ functions, $1<p<2d/(d+1)$,
in terms of Lebesgue space norms  for Fourier  localized pieces of the convolution kernel.
This is a special case of  corresponding results for general
Hankel multipliers. Besides $L^p-L^q$ bounds we also characterize
 weak type inequalities and intermediate inequalities involving
Lorentz spaces. Applications  include results on
 interpolation of multiplier spaces. \end{abstract}

\maketitle

\section{Introduction}
The purpose of this paper is to study convolution operators with
radial kernels acting on radial $L^p$ functions in $\bbR^d$. We
are interested in the boundedness properties of such operators on
$L^p_\rad$, the space of radial $L^p$ functions. It turns out
(perhaps surprisingly) that for a large range of $p$ one can
actually prove  a characterization in terms of the convolution
kernel. Moreover we
 also obtain characterizations for the  weak type $(p,p)$
inequality, or, more generally, results involving the
interpolating Lorentz spaces $L^{p,\sigma}_\rad$ for $p\le
\sigma\le \infty$. Here $L^{p,\sigma}_\rad$ denotes the subspace
of radial functions of the Lorentz space $L^{p,\sigma}(\bbR^d)$.
Recall that we have the strict inclusion $L^{p,\sigma_1}\subset
L^{p,\sigma_2}$ for $\sigma_1<\sigma_2$. The space
$L^{p,\infty}_\rad$ is the usual weak type $p$ space, and of
course
 $L^{p,p}_\rad=L^p_\rad$.

 Let $K\in \cS'(\bbR^d)$ be a radial convolution kernel, and denote by $\cT_K$ the convolution operator $f\mapsto \cT_K f=K*f$. We shall always
 assume that the Fourier transform $\widehat K$ 
is locally integrable;
this is a  trivial necessary condition for
 $L^p$ boundedness (and also for
 $L^p\to L^q$ boundedness with $q\le 2$).
Now consider the scaled kernels
$$K_t=t^{-d}K(t^{-1}\cdot).$$
Note that estimates for $\cT_K$ imply appropriately scaled
estimates for $\cT_{K_t}$, $t>0$.
 Let $\Phi$ be any radial
Schwartz function
whose
 Fourier transform is
compactly supported in $\bbR^d\setminus\{0\}$. By using dilation invariance and  testing the convolution with $K_t$  on
$\Phi$, we get a trivial necessary  condition for
$L^{p,1}_\rad\to L^{p,\sigma}$ boundedness of $\cT_K$, namely that
\Be \label{LpscondonK}
\sup_{t>0} \|\Phi*K_t\|_{L^{p,\sigma}}<\infty.\Ee
Our main result is that
\eqref {LpscondonK} for a single
nontrivial radial $\Phi$ is also sufficient for the convolution to
map $L^p_\rad$ to $L^{p,\sigma}$.

\begin{theorem}\label{main}
Let $K$ be  radial  and let  $\cT_K$ be the associated convolution operator.
Suppose $d>1$, $1<p<\frac{2d}{d+1}$, and $p\le \sigma\le\infty$.
Then the following statements are equivalent:

(a) There is a radial  Schwartz-function $\Phi$  (not identically zero)
for which  condition \eqref{LpscondonK} is satisfied.

(b) $\cT_K$ 
extends to a bounded operator  mapping  
 $L^{p,1}_\rad(\bbR^d)$ to $L^{p,\sigma}_\rad (\bbR^d)$.

(c)   $\cT_K$ extends to a bounded  operator mapping   
$L^{p}_\rad(\bbR^d)$ to $L^{p,\sigma}_\rad (\bbR^d)$.
\end{theorem}

As a consequence  one can
show that if in addition  $\widehat K$ is compactly supported away
from the origin
then the $L^p$ boundedness of   $\cT_K$ is equivalent
with $K\in L^p_\rad$.
{\it Cf.} \S\ref{compmult} for this and  somewhat
stronger results for boundedness on Lorentz spaces.
We remark that the condition  $p<2d/(d+1)$ is necessary  since for
$p\ge 2d/(d+1)$ there are radial $L^p$ kernels whose Fourier transforms
are  unbounded  and compactly supported in $\bbR^d\setminus\{0\}$,
{\it cf.} the comment following Corollary \ref{besovcor} below.


It is convenient to formulate these characterizations for
 more general Fourier-Bessel (or Hankel) transforms of
functions in $\bbR^+$.
As it is well known
(\cite{stw}, ch. IV) the Fourier transform of radial functions
can be expressed in terms  of integral
transforms on functions defined on $\bbR^+$, which is equipped
with the measure
$r^{d-1} dr$.
To be specific we define the Fourier transform  of  a
Schwartz function $g$ in $\bbR^d$ by
$\widehat g(\xi)\equiv \cF_{\bbR^d}[g](\xi)=\int g(y)e^{-i\inn{y}{\xi}} dy$.
We
recall  that if $g$ is {\it radial}, $g(x)= f(|x|)$ then its Fourier 
transform is radial and is given by
\Be\label{Fourierradial}\widehat g(\xi)= (2\pi)^{d/2} \cB_d f(\rho), 
\quad |\xi|=\rho,
\Ee
where $\cB_d$ denotes a Fourier-Bessel transform acting on functions
on the half line.
This
 transform can be defined for all {\it real} parameters $d>1$,
and it is given by
\Be\label{Fourier-Bessel}\cB_d f(\rho)= \int_0^\infty f(s)
B_d(s\rho) s^{d-1} ds
\Ee
where
\Be \label{Bd} B_d (\rho)= \rho^{-\frac{d-2}{2}} J_{\tfrac{d-2}2}(\rho)
\Ee
and $J_\alpha$ denotes the standard Bessel function.
This definition is closely related with the classical
(or {\it nonmodified})
Hankel transform given by
$$\cH_\alpha f(x)= \int_0^\infty \sqrt{xy} J_\alpha (xy) f(y) dy;$$
indeed
$\cB_d    =  M_{-\frac{d-1}{2}}
\cH_{\frac{d-2}{2}}M_{\frac{d-1}{2}}$
where the multiplication operator $M_c$ is defined by
$M_cf(r):=r^c f(r).$
The operator $\cB_d$ is just the {\it modified
Hankel transform} $H_\nu\equiv H_\nu^{\text {mod}}$
used in most papers on the subject, with the reparametrization
$H_\nu^{\text {mod}}= \cB_{2\nu+2}$.
We prefer our  notation  only because of the connection
with radial Fourier multipliers when $d$ is an integer. For $d=1$ one
recovers the cosine transform.
If $d>1$ is an integer then  the function $B_d$ in \eqref{Bd}
 represents (up to a constant) the Fourier transform of the surface measure on the unit sphere in $\bbR^d$. For
general $d\ge 1$  the functions $B_d$ are eigenfunctions with respect to the second order Bessel  differential operator $L=-D^2-\frac{d-1}\rho D$;  here $D=d/d\rho$.

In what follows
let
\Be \label{mud} d\mu_d = r^{d-1} dr\Ee and let
$L^p(\mu_d)$ be the Lebesgue space of measurable functions $f$
with $$\|f\|_{L^p(\mu_d)}=\Big(\int_0^\infty  |f(r)|^pr^{d-1} dr\Big)^{1/p}<\infty.$$
We continue to use the notation  $\|f\|_p$ for the standard $L^p$ norm on $\Bbb R$ (with respect to Lebesgue measure).
Let $\cS(\bbR_+)$ be the space of (restrictions to $\bbR^+$ of) 
even $C^\infty$ functions on $\bbR$ for which all derivatives decrease rapidly; then $\cB_d$ is
an isomorphism of $\cS(\bbR_+)$, an isometry of $L^2(\bbR_+,\mu_d)$,  and $\cB_d=\cB_d^{-1}$.
Clearly the space $\cS(\bbR_+)$ is dense in $L^p(\mu_d)$. It is also 
useful to note that the space $\cB_d(C^\infty_0)$ is 
dense in $L^p(\mu_d)$ for $1<p<\infty$; here $C^\infty_0$ is the class of
 $C^\infty$ functions with compact support in $(0,\infty)$. This statement is proved in Theorem 4.7 of \cite{st-tr}.
Clearly, if $m$ is locally integrable on $\bbR^+$
the operator $T_m$ defined by 
\Be\label{Tm}
T_m f(r)= \cB_d[ m\cB_d f](r)
\Ee is well defined
for $f\in \cB_d(C^\infty_0)$.
We remark that $L^1(\mu_d)$ is a commutative Banach algebra with
respect to a certain convolution structure \cite{guy}, and the
operators \eqref{Tm} can then be regarded as generalized convolutions.
However in this paper we shall not need to make
use of the precise  definition of the  convolution structure.

We now formulate necessary and sufficient characterizations for $L^p\to L^q$ boundedness for $T_m$ as well as extensions to Lorentz space
inequalities. 
Our main characterization
is in terms of
size properties of the
one-dimensional Fourier transform  of localizations of $m$.

\begin{theorem}\label{mainhankel}
Let $m\in L^1_\loc(\bbR^+)$ and
 let $\phi$ be a $C^\infty$ function
compactly supported in $\bbR_+$ (not identically zero).
Suppose
$1<d<\infty$, $1<p<\frac{2d}{d+1}$, $p\le q< 2$ and
$p\le \sigma\le \infty$.
Then the following statements are equivalent.

(i) $T_m$ extends to a bounded operator $T_m: L^{p,\omega}(\mu_d)\to
L^{q,\sigma}(\mu_d)$, for $\omega=\min\{\sigma, q\}$.

(ii)
$T_m$ extends to a bounded operator $T_m: L^{p,1}(\mu_d)\to
L^{q,\sigma}(\mu_d)$.


(iii)
\Be \label{besseld}\sup_{t>0} \,\,
t^{d(\frac 1p-\frac 1q)}
\big\|\cB_d[\phi m(t\cdot)]\big\|_{L^{q,\sigma}(\mu_d)}<\infty.
\Ee

(iv) With $k_t(x)=\cF^{-1}_\bbR[\phi m(t\cdot)](x)$,  the condition
\Be \label{global}\sup_{t>0} \,\,
t^{d(\frac 1p-\frac 1q)}\,
\big\| (1+|\cdot|)^{-\frac{d-1}{2} } k_t
\big\|_{L^{q,\sigma}((1+|x|)^{d-1} dx)} <\infty
\Ee
holds.
\end{theorem}

Condition \eqref{global} is simpler when $q=\sigma$, and in this case we see
that  $T_m$ is bounded
from $L^p(\mu_d)$ to $L^q(\mu_d)$
(and in fact from $L^{p,q}(\mu_d)$ to $L^q(\mu_d)$)
if and only if
\Be \label{Lqmudcond}
\sup_{t>0} \,\,
t^{d(\frac 1p-\frac 1q)}\,
\Big(\int_{-\infty}^\infty \big|k_t(x)\big|^q (1+|x|)^{(d-1)(1-\frac {q}{2})}
dx\Big)^{\frac 1q}<\infty;
\Ee
here again  $1<p<\frac{2d}{d+1}$  and now $p\le q\le 2$ (for the case $q=2$
see  \S\ref{Lp2est}).

Theorem \ref{main} is an immediate consequence of Theorem
\ref{mainhankel}. If $K=\cF^{-1}_{\bbR^d}[m(|\cdot|)]$ and
$g(x)=f(|x|)$ then $\cT_K g(x)=T_mf (|x|)$, by \eqref{Fourierradial}, and
the condition \eqref{besseld} is equivalent with
\Be\label{besseldradial}\sup_{t>0}\, t^{d(1/p-1/q)}
\big\|\cF_{\bbR^d}^{-1}[\phi(|\cdot|)m(t|\cdot|)]
\big\|_{L^{q,\sigma}(\bbR^d)}<\infty.
\Ee
Alternatively, after rescaling,  one can express this condition  using the
homogeneous  Besov type  space $\dot{B}_{-d/{p'},\infty}(L^{q,\sigma})$.
Namely for  radial $K$
 (with  $\widehat K\in L^1_\loc$)
and  $\Phi_t:=\cF^{-1} [\phi(t|\cdot|)]$,
\begin{align}
\notag
 \big\|\cT_K\big\|_{
L^p_{\rm
rad}(\mathbb{R}^d)\to L^{q,\sigma}_{\rm rad}(\mathbb{R}^d)}\,
\approx\,
\|K\|_{\dot{B}_{- d/{p'},\infty}(L^{q,\sigma})}&
\\ \approx
\sup_{t>0} t^{-d/p'} \big\|\Phi_{1/t}
*K\big\|_{L^{q,\sigma}}.&
\label{besovequivalence}
\end{align}
Note that the expression on the right hand side becomes a norm only
after considering the quotient of the space of distributions
modulo polynomials; however the (necessary) assumption that
$\widehat K$ is locally  integrable
 excludes polynomials
(and even nonzero constants). As a special case
(using the more familiar notation when $q=\sigma$)
the operator $\mathcal{T}_K$ maps boundedly $
L^p_{\rm
rad}(\mathbb{R}^d)\to L^{q}_{\rm rad}(\mathbb{R}^d)
$ if and
only if $\widehat K$ is locally integrable  and
$K\in \dot{B}^q_{-d/{p'},\infty}$.

We remark  that   no characterizations for $p\neq 1$ seem to
have been observed before; however
almost sharp results on compactly supported  multipliers
on $L^p_\rad(L^2_\sph)$ spaces on $\bbR^d$,
are in \cite{ms-adv}, in the sense that the exponent
$(d-1)(1-p/2)$ in \eqref{Lqmudcond} is replaced by $(d-1)(1-p/2)+\eps$.
Arai \cite{arai} proved a similar result with $\eps$-loss for global
Hankel multipliers, essentially
by combining arguments in \cite{ms-adv} and \cite{se-tams}.
We also note  that the necessity of the condition  \eqref{besseld} is trivial,
 and the necessity of conditions related to
 \eqref{global} is known from
\cite{gt-ind}, \cite{se-toh}, and
  \cite{arai};
{\it cf.} also  \S\ref{easy}
for an elementary  proof of the implication
$(iii)\implies(iv)$ in  Theorem  \ref{mainhankel}.
Finally note that 
Theorem \ref{mainhankel}
can be combined 
with  transplantation theorems for nonmodified Hankel transforms 
(\cite{guy}, \cite{st-tr}, \cite{stempak}, \cite{nst})
 to derive results on some other weighted $L^p$ spaces.

We state two consequences of the above 
characterizations concerning the structure of multiplier spaces.
It is convenient to
define
$\fM_d^{p,q}$, for $1<p\le q\le 2$
as  the space  of all locally integrable 
functions $m$ for which $T_m$ extends to a
bounded operator from
$L^p(\mu_d)$ to $L^q(\mu_d)$, and the norm is given by the
operator norm of $T_m$. 


A first  implication of  Theorem \ref{mainhankel} is that
 local multiplier conditions imply global ones; we state the case for 
$p=q$. Namely for nontrivial $\phi\in C^\infty_c(\bbR_+)$ one has the following equivalence.
\begin{corollary}\label{gl-loccorr} For $d>1$, $1<p<
\frac{2d}{d+1}$,
\Be\label{gl-loc}
\|m\|_{\fM^{p,p}_d}\approx \sup_{t>0} \|\phi m(t\cdot)\|_{\fM^{p,p}_d}.
\Ee
\end{corollary}
It is well known that the analogue of this corollary for $d=1$ and even
classes
of continuous
Fourier multipliers in $M^p$ on the real line is false, see   examples by Littman, McCarthy and Rivi\`ere \cite{lmr} and by Stein and Zygmund \cite{steinzyg}.


Another failing analogy to $M^{p}(\bbR)$ concerns the subject of
interpolation.
As a  straightforward consequence of the characterization we obtain an interpolation result
with respect to
the second complex
interpolation method  $[\cdot,\cdot]^\theta$,
introduced by Calder\'on (see \cite{cald}, and \cite{BL}, p.88).
In contrast, an extension of a result of Zafran (\cite{z}),
states that
the space $M^p(\bbR)$, $1<p<2$, is not an interpolation space for any pair $(M^{p_0},M^{p_1})$ with $p_0<p<p_1$, see
Appendix \S\ref{appendix}.

\begin{corollary} \label{interpolcor}
Suppose
$1<d_i<\infty$, $1<p_i<\frac{2d_i}{d_i+1}$,
 $p_i\le q_i\le 2$, for $i=0,1$, moreover that
$(1/p,1/q, d)=(1-\vth)(1/p_0,1/q_0,d_0)+\vth(1/p_1,1/q_1,d_1)$ with
 $0<\vth<1$. Then
\Be\label{interpol} [\fM^{p_0,q_0}_{d_0}, \fM^{p_1,q_1}_{d_1}]^\vth=
\fM^{p,q}_{d}.
\Ee
\end{corollary}
This result follows from
 interpolation  of certain
Fourier-localized versions of weighted $L^p$ spaces (which are defined by
\eqref{global}),  see  Lemma \ref{interpolfloc} below.
For  a related result on real interpolation see \S\ref{compmult}.

Finally by standard arguments using H\"older's inequality and Plancherel's theorem  condition \eqref{global} implies the
known sufficient criteria
of H\"ormander type (\cite {gt-ma}), which are formulated
 using localized $L^2$-Sobolev spaces; these were
termed $S(2,\alpha)$ in \cite{connschw} and
 $WBV_{2,\alpha}$ (with $\alpha>1/2$) in \cite{gt-studia}.
The following endpoint
bounds in terms of
localized versions of Besov spaces
seem to be new; it is an optimal estimate within the class of
$L^2$-smoothness assumptions.
Recall $\|g\|_{B^2_{a,q}}\approx (\sum_{k=0}^\infty
2^{k aq}\|\widehat g\|_{L^2(\cI_k)}^q)^{1/q}$ where $\cI_0=[-1,1]$ and
$\cI_k=\{\xi\in \bbR: 2^{k-1}\le |\xi|\le 2^k\}$, for
$k>1$.

\begin{corollary} \label{besovcor}
For $1<d<\infty$, $1<p<\frac{2d}{d+1}$, $p\le q\le 2$,
\Be\label{besov}
\|m\|_{\fM^{p,q}_d} \lc \sup_{t>0}\, t^{d(\frac{1}{p}-\frac {1}{q})}\|\phi m(t\cdot)\|_{B^2_{a,q}}, \quad a=d(\frac 1q-\frac 12).
\Ee
\end{corollary}
Here, and in what follows, the notation $\lc$  indicates that in the inequality an unspecified constant is involved which may depend on $d,p,q$.
Since the space $B^2_{1/2, p}$ contains unbounded functions for $p>1$ the
corollary does not extend to the endpoint $p=q=2d/(d+1)$.

\noi{\it This paper.}
In \S\ref{prel} we gather various facts on
Bessel functions, Littlewood-Paley inequalities,
interpolation
and elementary convolution inequalities on weighted spaces, needed later in the paper.
In \S\ref{kernelestimates}
we derive some pointwise bounds for the
kernels  of  multiplier transformations, assuming that
 the multipliers are compactly supported in $(1/2,2)$.
In \S\ref{easy} we prove the necessity of the conditions, namely the
implications ${\rm(i)}\!\implies\!{\rm(ii)}
\!\implies\!{\rm(iii)} \!\implies\!{\rm(iv)}$ of Theorem
\ref{mainhankel}. The proof of the main implication
${\rm(iv)}\!\implies\!{\rm(i)}$ is contained in
sections \ref{sufficiency}-\ref{interpolsect}.
In \S\ref{sufficiency} we discuss the basic decomposition into Hardy type and singular integral operators.
The crucial estimate for the main Hardy-type operator is proved in
\S\ref{Hardybounds}, and \S\ref{sing-hardy} contains estimates for
better behaved operators (in particular singular integrals) for which we do not need the full strength of assumption \eqref{global}.
In \S\ref{Lp2est} we give the straightforward proof of the
$L^p\to L^2$ bounds and then conclude in
\S\ref{interpolsect}
the proof of  the implication
${\rm(iv)}\!\implies\!{\rm(i)}$ by an interpolation.
In \S\ref{compmult} we
 give the  short proofs of the  Corollaries and briefly discuss a
further result on real interpolation and an improved   version
of our results for
multipliers which are compactly
supported away from the origin.
 Some open problems are mentioned in \S\ref{open}. An appendix
(\S\ref{appendix}) is included with the above mentioned non-interpolation results for
Fourier multipliers.

\section{Preliminaries} \label{prel}


\medskip
\noi{\bf Asymptotics for Bessel functions.} In order to relate the
Hankel transforms of multipliers to the one-dimensional Fourier
transform we need to use standard  asymptotics for Bessel
functions (see \cite{erdelyi}, 7.13.1(3)), namely for $|x|\ge 1$,
\begin{multline*}
B_d (x)=\sum_{\nu=0}^M c_{\nu,d}
\cos (x-\tfrac{d-1}{4} \pi) x^{-2\nu- \frac{d-1}2}
\\+\sum_{\nu=0}^M \widetilde c_{\nu,d}
\sin(x-\tfrac{d-1}{4} \pi) x^{-2\nu- \frac{d+1}2}+ x^{-M}
\widetilde E_{M,d}(x)
\end{multline*}
with $c_{0,d}=(2/\pi)^{1/2}$,
and the derivatives of
$\widetilde E_{M,d}$ are bounded.
Thus one may also write down expansions for the derivatives and, after
writing the cosine and sine terms as combinations of
exponentials and applying the previous formula with  $M$ replaced by
$M+k$ one  also gets, for $|x|\ge 1$,
\Be
\label{besselasympt}
B_d^{(k)} (x)=\sum_{\nu=0}^M (c_{\nu,k,d}^+ e^{ix}+ c_{\nu,k,d}^- e^{-ix})
 x^{-\nu- \frac{d-1}2}
+ x^{-M} E_{M,k,d}(x)
\Ee where $c_{0,0,d}^\pm= (2\pi)^{-1/2} e^{\mp i\frac{d-1}{4}\pi}$
and the  $E_{M,k,d}$ have bounded derivatives:
\Be \label{derivativesofE}
 |E_{M,k,d}^{(k_1)}(x)|\le C(M,k,k_1,d).
\Ee

\noi{\bf Littlewood-Paley inequalities.}
Let $\eta\in C^\infty(\bbR_+)$ with compact support away from $0$.
Let $L_j f= \cB_d[\eta(2^{-j}\cdot) \cB_d f]$.
Then for $1<p<\infty$ there are  the inequalities
\begin{align}
\label{LP1}
&\Big\|\Big(\sum_{j\in \bbZ}|L_j f|^2 \Big)^{1/2}\Big\|_{L^{p}(\mu_d)}
\le C_p \|f\|_{L^{p}(\mu_d)},
\\
\label{LP2}
&\Big\|\sum_{j\in \bbZ}L_j f_j\Big\|_{L^p(\mu_d)} \le C_{p}' 
\Big\|\Big(\sum_{j\in \bbZ}|f_j|^2\Big)^{1/2}\Big\|_{L^{p}(\mu_d)};
\end{align}
indeed \eqref{LP1} and \eqref{LP2} are dual to each other with
$C_{p}'=C_{p'}$, $1/p+1/p'=1$.
By the real
(Lions-Peetre) interpolation method the spaces $L^p(\mu_d)$ can be replaced by
$L^{p,\sigma}(\mu_d)$, for any  $\sigma$.

For the proof of \eqref{LP1}, \eqref{LP2} we note that the operators
\[ f\mapsto  \sum_j \pm L_j f\]
are bounded on $L^p(\mu_d)$, $1<p<\infty$, with operator norm independent of the choice of signs $\pm$. This follows for example by (a non-sharp version of)
the H\"ormander type multiplier criterion for modified Hankel transforms in
Gasper and Trebels
\cite{gt-ma}; for the case of integer $d$ one could simply use standard results in $\Bbb R^d$ specialized to radial functions (\cite{stein-si}).
Now the inequalities  \eqref{LP1}, \eqref{LP2}
follow by the usual
 averaging argument using Rademacher functions
(see \cite{stein-si}, ch. IV, \S5.2), and a duality argument.

\medskip

\noi{\bf Remarks on Lorentz spaces.} We assume that $\Omega$ is a measure
space with given $\sigma$-algebra and underlying measure $\mu$.
We refer to a thorough discussion of Lorentz spaces to
\cite{stw}. There the definition of $L^{q,\sigma}$  is given in terms of rearrangements of
$f$ and it is shown that this definition is
equivalent to a norm when
$1<q<\infty$, $1\le \sigma\le \infty$.
Instead of the rearrangement function one can also use the distribution function and it is easy to check (on simple functions) that an equivalent
quasi-norm on $L^{q,\sigma}$ is given by
\Be\label{quasinormlor}
\|f\|_{L^{q,\sigma}}\approx \Big(\sum_{\ell=-\infty}^\infty 2^{\ell\sigma}
\big[ \mu\big(\{ x\in\Omega: |f(x)|>2^\ell\}\big)\big]^{\sigma/q}\Big)^{1/\sigma}
\Ee
(with the natural $\ell^\infty$ analogue for $\sigma=\infty$).
For the manipulation of vector-valued functions we shall need the following inequality.
\begin{lemma}\label{vectorlor}  Let $1<q<r$, $1\le \sigma\le \infty$ and
let $\{F_j\}$ be a sequence of measurable functions on $\Omega$. Then
\Be \Big\|\Big(\sum_j|F_j|^r\Big)^{1/r}\Big\|_{L^{q,\sigma}}
\le C(q,\sigma,r) \Big(\sum_j \big\|F_j\big\|^\om_{L^{q,\sigma}}\Big)^{1/\om},
\quad \omega=\min\{\sigma, q\}.
\Ee
\end{lemma}
\begin{proof} Consider measurable functions $H$ on $\Omega \times \bbZ$.
We first claim that for
$1<q<r$, $1\le \sigma\le \infty$
\Be \label{productlor}
\Big\|\Big(\sum_j|H(\cdot, j)|^r\Big)^{1/r}
\Big\|_{L^{q,\sigma}(\Omega)}
\le c(q,\sigma,r) \|H\|_{L^{q,\sigma}(\Omega\times \bbZ)}.
\Ee
For the case $q=\sigma$ this follows by applying the imbedding
$\ell^q\hookrightarrow \ell^r$ and then Fubini's theorem
(interchanging a sum and an integral). For arbitrary $\sigma$ it
follows by applying the real
 method of interpolation.
Now we apply  \eqref{quasinormlor} to the right hand side of
\eqref{productlor} and estimate
for $\sigma\ge q$
\begin{align*}
&\|H\|_{L^{q,\sigma}(\Omega\times \bbZ)}
\lc
\Big(\sum_\ell 2^{\ell \sigma}
\Big(\sum_j \mu\big( \{x:|H(x,j)|>2^{\ell}\}\big)
\Big)^{\sigma/q} \Big)^{1/\sigma}\lc
\\&
\Big(\sum_j \Big(\sum_{\ell} 2^{\ell\sigma}
\mu\big( \{x:|H(x,j)|>2^{\ell}\}\big)^{\sigma/q} \Big)^{q/\sigma}\Big)^{1/q}
 \lc\Big(\sum_j\big\|H(\cdot,j)\big\|_{L^{q,\sigma}}^q\Big)^{1/q};
\end{align*}
here we have used Minkowski's  inequality for the sequence space
$\ell^{\sigma/q}$.
If $ \sigma<q$ we use instead the imbedding $\ell^{\sigma/q}\subset \ell^1$
and
estimate $ \|H\|_{L^{q,\sigma}(\Omega\times \bbZ)}$ by
\begin{align*}
\Big(\sum_\ell 2^{\ell \sigma}
\sum_j \big(\mu\big( \{x:|H(x,j)|>2^{\ell}\}\big)\big)^{\sigma/q}
\Big)^{1/\sigma}
\approx
 \Big(\sum_j\big\|H(\cdot,j)\big\|_{L^{q,\sigma}}^\sigma\Big)^{1/\sigma}.
\end{align*}
\end{proof}

\medskip

\noi{\bf Elementary inequalities for weighted norms.} To handle expressions such as
\eqref{global} we need some elementary
inequalities on convolutions and dilations.

\begin{lemma}\label{straightfconv}
Let $a\ge 0$, and
$\ga>a+1$. Suppose that $g$, $\zeta$ are Lebesgue measurable on $\bbR$
 and
$\zeta$
 satisfies
\Be |\zeta(x)|\le C_1 (1+|x|)^{-\gamma}.
\label{zetaest}\Ee
Then for $q_1\ge q\ge 1$
\Be\label{convwithSchw}
\Big(\int |g*\zeta(x)|^{q_1} (1+|x|)^{aq_1} dx\Big)^{1/q_1}\lc C_1
\Big(\int |g(x)|^q (1+|x|)^{aq} dx\Big)^{1/q}.
\Ee
Also
\Be\label{effectofdilation}
\Big(\int |g(tx)|^q (1+|x|)^{aq} dx\Big)^{1/q}
\le t^{-1/q}\max\{1, t^{-a}\}\Big(\int |g(x)|^q (1+|x|)^{aq} dx\Big)^{1/q}.
\Ee
\end{lemma}

\begin{proof}
For $q=q_1$ the left hand side of \eqref{convwithSchw}
is dominated by a constant times 
\begin{align*}&\int(1+|y|)^{-\ga}\Big(\int|g(x-y)|^q (1+|x|)^{aq} dx\Big)^{1/q}dy
\\&\le\int(1+|y|)^{-\ga+a}dy\,\Big(\int|g(x)|^q (1+|x|)^{aq} dx\Big)^{1/q}
\end{align*}
where we have used $1+|x|\le (1+|x-y|) (1+|y|)$.
The integral is finite since $\ga>a+1$.

The analogue of \eqref{convwithSchw} for
 $q_1=\infty$  is also valid; we estimate (assuming momentarily $q>1$)
\begin{multline*}
|g*\zeta(x)|(1+|x|)^a\lc   \\
\Big(\int|g(x-y)(1+|x-y|)^{a}|^q dy\Big)^{1/q}
\Big(\int \frac{(1+|y|)^{-\gamma q'}(1+|x|)^{aq'}}
{(1+|x-y|)^{aq'} }dy\Big)^{1/q'}
\end{multline*}
where the first term is the desired expression on the right hand side of
\eqref{convwithSchw} and the second term is
$\lc(\int (1+|y|)^{(a-\gamma) q'}dy)^{1/q'}$, hence finite.
A similar argument holds for $q=1$.
We have now proved the asserted  bound for $q_1=\infty$  and $q_1=q$ and the intermediate cases follow by interpolation.

Inequality \eqref{effectofdilation} follows from
$(1+|x|/t)\le \max\{t^{-1},1\}(1+|x|)$ and
a change of variable.
\end{proof}

We shall need the following  Lorentz space variant
of Lemma  \ref{straightfconv}
which will be used repeatedly.

\begin{lemma}\label{straightfconvLor}
Let  $\alpha> \beta q\ge 0$,
$1< q<\infty$, $1\le \sigma\le \infty$, and
let $d\nu_\alpha= (1+|x|)^{\alpha}dx$ (as a measure on $\bbR$).
Suppose that
$\zeta$
 satisfies \eqref{zetaest} for some $\gamma> 1-\beta+\alpha /q $.
Then
\Be\label{convwithSchwLor}
\Big\|\frac{g*\zeta}{(1+|\cdot|)^\beta}\Big\|_{L^{q,\sigma}(\nu_\alpha)}
\lc \Big\|\frac{g}{(1+|\cdot|)^\beta}\Big\|_{L^{q,\sigma}(\nu_\alpha)}
\Ee
and
\Be\label{effectofdilationLor}
\Big\|\frac{g(t\cdot)}{(1+|\cdot|)^\beta}\Big\|_{L^{q,\sigma}(\nu_\alpha)}
\lc t^{-1/q}\max\{1, t^{-\alpha/q+\beta}
\Big\|\frac{g}{(1+|\cdot|)^\beta}
\Big\|_{L^{q,\sigma}(\nu_\alpha)}
\Ee
\end{lemma}

\begin{proof}  Define
$\cM_\beta f:=(1+|x|)^{\beta} f(x)$ and let $S_\zeta f(x)=\zeta*f$. Then the assertion is equivalent with the claim that
$\cM_{-\beta} S_\zeta
\cM_\beta$
is bounded on  $L^{q,\sigma}(\nu_\alpha)$. Since $1<q<\infty$ and restriction on $\gamma$ also involves a strict inequality the general Lorentz space estimate follows from the case $q=\sigma$ by real interpolation.
The $L^q(\nu_\alpha)$ boundedness of
$\cM_{-\beta} S_\zeta
\cM_\beta$
 is in turn equivalent to the inequality \eqref{convwithSchw}
for the choice $q=q_1$ and $aq= \alpha-\beta q$. We may apply
\eqref{convwithSchw}  since
$\gamma>a+1=1-\beta+\alpha/q$. The proof of \eqref{effectofdilationLor}
is similar.
\end{proof}

\noi{\bf Independence of the localizing function.}
Let $a\ge 0$,  $b\in \bbR$, $1\le p\le 2$.
Let $\phi$ be a smooth function supported on a compact subinterval of $(0,\infty)$, and assume that $\phi$ is not identically zero.
It will be
 convenient to denote by  $\floc(p,a,b)$ the space of all $m$
which are  integrable  over every compact subinterval of $(0,\infty)$
and satisfy the condition
\Be\label{floc}
\sup_{t>0} t^b \Big(\int_{-\infty}^\infty \big|\cF_\bbR^{-1}[\phi m(t\cdot)](x)\big|^p (1+|x|)^{ap} dx\Big)^{1/p}\le A
\Ee for some finite $A$. Here $\floc$
refers to localization and to the Fourier transform.

We use Lemma \ref{straightfconv} and
Lemma \ref{straightfconvLor}
to prove that the
choice   of the  cutoff function $\phi$ in \eqref{floc} and
 \eqref{global}  does not matter. Moreover we wish,
for suitable $\phi$,  use discrete conditions where  the sup is taken
 over dyadic $t$. To formulate these
choose  $\vphi\in C^\infty_c(\frac 12, 2)$
with the property that
\Be\label{speccutoff}
\sum_{j\in \bbZ} \vphi^2(2^{-j}s)=1, \quad s>0.\Ee

\begin{lemma}\label{indepofphi}
Let $1<q<\infty$, $1\le \sigma\le \infty$.

(i)
 Suppose
\Be\label{floclor} \sup_{t>0} t^b \Big\|\frac{|\cF_\bbR^{-1}[\phi
m(t\cdot)]}
{(1+|\cdot|)^\beta}\Big\|_{L^{q,\sigma}(\nu_\alpha)}\le A<\infty
\Ee  holds for some $\phi\in C^\infty_c(\bbR^+)$ which is not
identically zero. Let $\eta\in C^\infty_c(\bbR^+)$. Then the
expression analogous to \eqref{floc}, but with $\phi$ replaced by
$\eta$, is bounded by $CA$, where $C$ does not depend on $m$.

(ii)  With  $\vphi\in C^\infty_c(\frac 12, 2)$
satisfying \eqref{speccutoff}  the left hand side of
\eqref{floclor} is bounded by
\Be\label{floclordisc}
C\sup_j 2^{jb}
\Big\|\frac{|\cF_\bbR^{-1}[\vphi m(2^j\cdot)]}
{(1+|\cdot|)^\beta}\Big\|_{L^{q,\sigma}(\nu_\alpha)}.
\Ee
\end{lemma}

\begin{proof}
We begin by observing that
$\int_0^\infty \phi^2(\tau s) \frac{d\tau}\tau= c>0$ independent of $s$.
Hence $\eta(s) m(ts)=c^{-1}\int_0^\infty
 \phi^2(\tau s)\eta(s) m(ts)\tau^{-1}d\tau$
and since  if $s$ is taken from a compact subset  of $(0,\infty)$
the integral reduces to an integral over $[\eps, \eps^{-1}]$ for some $\eps\in (0,1)$. Thus
\[\cF^{-1}[\eta(s) m(ts)]=
\int_{\eps}^{1/\eps} \Phi_\tau * \big(\tau^{-1}k_{t/\tau}(\tau^{-1}
\cdot)\big) \frac{d\tau}{\tau}
\]
where $\Phi_\tau=\cF^{-1}[
 \phi(\tau \cdot) \eta]$
and $k_t=\cF^{-1}[
 \phi m(t\cdot)]$.
Now the assertion (i) follows
immediately from
\eqref{convwithSchwLor}
and \eqref{effectofdilationLor}.
(ii) is proved similarly; the details are left to the reader.
\end{proof}


\medskip
\noi{\bf Interpolation.}
Interpolation results for the spaces $\floc(p,a,b)$
are analogous to those for localized potential spaces in
\cite{connschw},
\cite{cgt}, with a very similar proof; therefore we only give a sketch.
We denote by $[\cdot,\cdot]_\vth$, $[\cdot,\cdot]^\vth$
the complex interpolation methods of Calder\'on
(see \cite{cald}, and  also ch. 4 in \cite{BL}).

\begin{lemma} \label{interpolfloc}
Let $a_0, a_1\ge 0$, $b_0, b_1\in \bbR$ and
$1\le p_0, p_1, \le 2$. Suppose that
$(a,b,p^{-1})=(1-\vth)(a_0,b_0,p_0^{-1})+\vth
(a_1,b_1,p_1^{-1})$ and $0<\vth<1$. Then
\Be \label{interpolfloceq}
\big[ \floc(p_0,a_0,b_0), \floc(p_1,a_1,b_1)\big]^\vth=
\floc(p,a,b).
\Ee
\end{lemma}

\begin{proof}[Sketch of proof]
Let
$\|K\|_{L(p,a)}:=(\int_{-\infty}^\infty|K(t)|^p (1+|t|)^{ap} dt)^{1/p}$
and denote by  $\ell^\infty_b(L(p,a))$ be the space of sequences of $L(p,a)$ functions $\{G_j\}_{j\in \bbZ}$ for which
$\sup_j 2^{jb} \|G_j\|_{L(p,a)}<\infty$.
Weighted $L^p$ spaces can be interpolated by the complex method
(see \cite{BL}, ch. 5) and we have
$$L(p,a)=[L(p_0,a_0), L(p_1,a_1)]_\vth.$$
By a result of Calder\'on  (\cite{cald}, \S 13.6)
\Be \label{seqinterpol}
\ell_{b}^\infty(L(p,a))
=[\ell_{b_0}^\infty(L(p_0,a_0)),
\ell_{b_1}^\infty(L(p_1,a_1))]^\vth
\Ee
and one has to show that $\floc (p,a,b)$ is a retract of
$\ell_{b}^\infty(L(p,a))$; {\it i.e.} there are bounded linear  operators
\begin{align*}
&\fA: \floc(p,a,b)\to \ell_b^\infty(L(p,a))\, ,
\\
&\fB: \ell_b^\infty(L(p,a)) \to \floc(p,a,b)\, ,
\end{align*}
so that $\fB\circ\fA$ is the identity on $\floc(p,a,b)$.
These maps are given by
\begin{align}
\big[\fA m\big]_j &= \cF^{-1}[\vphi m(2^j\cdot)]\,, \label{Adef}
\\
\fB G &= \sum_{k\in \bbZ} \vphi(2^{-k}\cdot) \widehat{G_k} (2^{-k}\cdot)\,.
\label{Bdef}
\end{align}
$\fA$ is bounded by definition of the $\floc(p,a,b)$ norm and the
boundedness  of $\fB$ is straightforward; one uses Lemma \ref{straightfconv}.
Also
$\fB\circ\fA$ is the identity on $\floc(a,b,p)$, by \eqref{speccutoff}.
This shows \eqref{interpolfloceq},  the details are left to the reader.
\end{proof}

\noi{\it Remark.} The analogues of these theorems
for localized potential spaces are proved by Connett and Schwartz in
\cite{connschw}, see also \cite{cgt}.
In \cite{connschw} it is also noted that the analogue of
\eqref{interpolfloceq} fails for the $[\cdot,\cdot]_\vth$ method (and their argument applies here as well).
In addition, if
$\floc_o(p,a,b)$
denotes the closed subspace of functions
for which
the expressions $2^{jb}
\|\cF_\bbR^{-1}[\vphi m(2^j\cdot)]\|_{L^p((1+|x|)^{ap} dx)}$
tend to $0$ as $|j|\to \infty$,
then one also has
$\big[ \floc_o(p_0,a_0,b_0), \floc_o(p_1,a_1,b_1)\big]_\vth=
\floc_o(p,a,b)$. This is analogous to a result in \cite{connschw} on localized potential spaces.

%

\section{Kernel estimates}\label{kernelestimates}
Assume that
 the multiplier $m$ has compact support in $[\frac 12,2]$. Here we give pointwise estimates for the
kernel of multiplier transformations involving two Bessel
transforms $\cB_a$, $\cB_b$ of possibly different orders; however
the main interesting case is of course $a=b=d$. We can write for
$a,b>0$
\Be \label{integralrep}\cB_a[m\cB_b f](r)= \int \cK_{a,b}[m](r,s)
s^{b-1} f(s) ds \Ee where the kernel is given by
\Be\label{kernelintegral} \cK_{a,b}(r,s)\equiv
\cK_{a,b}[m](r,s)=\int_0^\infty m(\rho) B_a(\rho r)B_b(\rho s)
\rho^{a-1} d\rho. \Ee

\begin{proposition}\label{kernest}
Let $a\ge1$, $b\ge 1$, $N >1$ and let  $m$ be integrable and supported
 in $[\frac 12,2]$.
Then  for $\beta, \gamma=0,1,2,\dots$
\begin{multline} \label{kernestim}
\big|
\partial_r^{\beta} \partial_s^\gamma \cK_{a,b}[m](r,s)\big| \le
\\
C_{N,\beta,\gamma} \sum_{(\pm,\pm)} (1+r)^{-\frac{a-1}{2}}
(1+s)^{-\frac{b-1}{2}} \int\frac {|\cF_\bbR^{-1}[m](\pm r\pm
s-u)|}{(1+|u|)^N} du.
\end{multline}
\end{proposition}





\begin{proof}
We begin with a preliminary observation, which we shall use several times, namely the inequality
\Be \label{elemconv}
(1+R)^{-M}\int \frac{|g(u)|}{(1+|u|)^{N_1}}du
\le C (1+R)^{-M+N_1} \int \frac{|g(R+u)|}{(1+|u|)^{N_1}}du;
\Ee
this is (similar to the statement in Lemma \ref{straightfconv})
 a consequence of the triangle inequality and a translation in the integral.

 Let $\chi$ be a $C^\infty$ function so that
$\chi(s)=1$ for $s\in (1/2,2)$ and $\chi$ is supported in
$(1/4,4)$. If $r,s\le 1$ then the function
\[\rho\mapsto  h(\rho)=\chi(\rho)\rho^{a-1+\beta+\gamma} B_a^{(\beta)}(\rho
r) B_b^{(\gamma)}(\rho s)\] is smooth and has a rapidly decaying
Fourier transform, with bounds uniform in $r,s\le 1$. Denote the
Fourier transform by $u\mapsto \la(r,s,u)$. We may apply
duality for the Fourier transform 
and estimate (with  $\ka=\cF_\bbR^{-1}[m]$)
\Be
\label{plest}
|\partial_r^{\beta} \partial_s^\gamma \cK_{a,b}[m](r,s)|
=\Big|\int \kappa(u){\la(r,s,u)} du\Big|
 \le
C_{N_1,\beta,\gamma}  \int \frac{|\kappa(u)|}{(1+|u|)^{N_1}} du.
\Ee
Clearly this term is bounded by a suitable  constant times any  of the terms
on the right hand side of \eqref{kernestim}, as long as $r,s\le 1$.

Next we consider the case $s\le 1$, $ r\ge 1/2$ and use the
asymptotic expansion \eqref{besselasympt} for $B_a(\rho r)$ and
its derivatives. We assume that the parameter $M$ is chosen large,
in order to use \eqref{elemconv}, in fact we require
$M>2N+(a+b)/2$.

This yields
\begin{align*}
\partial_r^{\beta} \partial_s^\gamma 
\cK_{a,b}(r,s)&=\sum_{\pm}\sum_{\nu=0}^M r^{-\frac{a-1}{2}-\nu}
\int m(\rho) e^{\pm ir\rho}
 \eta^{\pm}_{\nu, \beta,a,b}(s,\rho) d\rho
\\&+r^{-M}
\int m(\rho) \om_{M,\beta,\gamma,a,b}(r,s,\rho) d\rho
\end{align*}
where
\begin{align*} \eta^{\pm}_{\nu, \beta,a,b}(s,\rho) &=c^\pm_{\nu, \beta,a}\chi(\rho)
\rho^{\frac{a-1}{2}+\beta+\gamma-\nu} B_b^{(\gamma)}(s\rho)\, ,
\\
\om_{M,\beta,\gamma,a,b}(r,s,\rho)&=\rho^{-M+\beta+\gamma+a-1}
E_{M,\beta,a}(r\rho) B_b^{(\gamma)}(s\rho)\, .
\end{align*}

The terms in the sum can be realized as convolutions of $\kappa$
with rapidly decaying functions, multiplied with
$r^{-\frac{a-1}{2}-\nu}$. These terms are bounded by
\[r^{-\frac{a-1}2}\int\frac {|\kappa(\mp r-u)|}{(1+|u|)^N}du\]
and since $s\lc 1$ this also implies the bound by the sum of terms on the right hand side of \eqref{kernestim}.
For the error term we argue as above, using duality  to estimate
\[
r^{-M}
\Big|\int \ka(u) {\widehat \om_{M,\beta,\gamma,a,b}(r,s,u)}
du\Big|
\lc r^{-M+N} \int\frac{|\kappa(u)|}{(1+|u|)^N}du\]
and the desired estimate  follows from using \eqref{elemconv},
recall $M>2N+(a-1)/2$.

The estimations for the case $r\lc 1$ and $s\gc 1$ are similar, the roles
 of $r$ and $s$ are reversed.

Finally, to handle the case $r, s\ge 1/2$  we use the asymptotic
expansion \eqref{besselasympt} for both $B_a(\rho s)$ and
$B_b(\rho r)$, again with large $M$. We then write
\[\partial^\beta_r\partial^\gamma_s \cK_{a,b}(r,s) =\int_0^\infty m(\rho)
\rho^{a-1+\beta+\gamma} B_a^{(\beta)}(r\rho) B_b^{(\gamma)}(s\rho)
d\rho\] as
\begin{align}
\label{fourterms} & \sum_{\nu,\nu'}\sum_{\pm,\pm} c_{\nu,
\beta,a}^\pm c_{\nu', \gamma,b}^\pm r^{-\frac{a-1}{2}-\nu}
s^{-\frac{b-1}{2}-\nu'} \int  m(\rho)
\rho^{\frac{a-b}{2}+\beta+\gamma-\nu-\nu'} e^{i\rho(\pm r\pm
s)}d\rho
\\&+
\sum_{\nu}\sum_{\pm} c_{\nu, \beta,a}^\pm
r^{-\frac{a-1}{2}-\nu}s^{-M} \int  m(\rho)
\rho^{\frac{a-1}{2}+\beta+\gamma-\nu-M} E_{M,\gamma,b}(\rho s)
 e^{\pm i\rho r}d\rho
\notag
\\&+
 \sum_{\nu'}\sum_{\pm}
c_{\nu', \gamma,b}^\pm s^{-\frac{b-1}{2}-\nu'}r^{-M} \int  m(\rho)
\rho^{a-\frac{b+1}{2}+\beta+\gamma-\nu'-M} E_{M,\beta,a }(\rho r)
e^{\pm i\rho s}d\rho \notag
\\&+ (rs)^{-M}
\int  m(\rho) \rho^{a-1+\beta+\gamma-2M}
 E_{M,\beta,a}(\rho r) E_{M,\gamma,b}(\rho s) d\rho.
\notag
\end{align}
The first (double) sum in \eqref{fourterms} is clearly bounded by the right hand side of \eqref{kernestim}.
The second, third and fourth terms are bounded, by the
previous arguments by a constant times
\[r^{-\frac{a-1}2}s^{N-M}\int\frac{|\ka(\mp r+u)|}{(1+|u|)^{N}}du, \quad
s^{-\frac{b-1}2}r^{N-M}\int\frac{|\ka(\mp s+u)|}{(1+|u|)^{N}}du,
\]
and $(rs)^{N-M}\int|\ka(u)|(1+|u|)^{-N}du,$ respectively. However
by using inequality \eqref{elemconv} and the condition
$M>2N+(a+b)/2$ these terms are seen to be also bounded by the
right hand side of \eqref{kernestim}.
\end{proof}

Proposition \ref{kernest} is mainly interesting as an estimate for
general multipliers. However for the proof of necessary conditions
we record a straightforward consequence for smooth multipliers, in
the
special case where $a=1$, $b=d$.

\begin{corollary} \label{babytranspl}
Let $d\ge 1$ and let $\chi\in C^\infty$ be supported in $[1/4,4]$.
Then for any  $M\ge 0$
\begin{equation*}
|\cB_1[\chi \cB_d f](r)|\le C_M \int_0^\infty \frac{|f(s)|
s^{d-1}}{(1+|r-s|)^M (1+s)^{\frac{d-1}{2}}} ds\,.
\end{equation*}
\end{corollary}
\begin{proof}
  We use the estimate of Proposition \ref{kernest} in conjunction with a simple convolution inequality which is based on the rapid decay of
$\cF^{-1}[\chi]$.
\end{proof}

\section{The  implications
{\rm (i)$\implies$(ii)$\implies$ (iii)$\implies$(iv)}
of Theorem \ref{mainhankel}}
\label{easy}
%
\begin{proof}[Proof of $(i)\implies (ii)$] This follows from
$L^{p,1}(\mu_d)\subset L^{p,\sigma}(\mu_d)$, for $\sigma\ge 1$,
with continuous imbedding.
\end{proof}

\begin{proof}[Proof of $(ii)\implies (iii)$]
We use the dilation formula
\Be\label{dil}
\cB_d[g(t^{-1}\cdot)](r)=t^d \cB_d[g](tr).
\Ee
If $\phi\in C^\infty_c(\bbR_+)$ then the function $f_1:=\cB_d\phi$ belongs to
$L^{p,1}(\mu_d)$ for all $p$
and has positive norm. Now set
$f_t=t^{-d(1-1/p)}\cB_d[\phi(t^{-1}\cdot)]$; then the $L^{p,1}(\mu_d)$
norm of $f_t$ is independent of $t$. Let $\|m\|$ denote the
$L^{p,1}(\mu_d)\to L^{q,\sigma}(\mu_d)$ operator norm of $T_m$.
We may estimate
\begin{align*}
&\|f_1\|_{L^{p,1}(\mu_d)}\|m\|=
\|f_t\|_{L^{p,1}(\mu_d)}\|m\|
\ge \|\cB_d[m \cB_d f_t]\|_{L^{q,\sigma}(\mu_d)}\\&=t^{-d(1-1/p)}
\|\cB_d[\phi(t^{-1}\cdot)m]\|_{L^{q,\sigma}(\mu_d)}
= t^{d(1/p-1/q)} \big\|\cB_d[m(t\cdot)\phi]\big\|_{L^{q,\sigma}(\mu_d)}
\end{align*}
which proves the implication.
\end{proof}

\begin{proof}[Proof of $(iii)\implies (iv)$]
Let $u_\ev(t,\rho)$ the even extension of $\phi(\rho)m(t\rho)$ to
$\bbR$. Let $h_t:=\cF^{-1}_\bbR[u_{\ev}(t,\cdot)]$. We claim that
it suffices to show
\Be\label{ueven} \big\|(1+|\cdot|)^{-\frac{d-1}{2}}
h_t\big\|_{L^{q,\sigma}(\nu)} \lc  \big\|\cB_d [\phi
m(t\cdot)]\big\|_{L^{q,\sigma}(\mu_d)}, q<2, \Ee
where $d\nu(x)= (1+|x|)^{d-1} dx$.
Indeed if \eqref{ueven} holds let
 $\zeta\in \cS(\bbR)$ so that $\widehat \zeta  $ is
supported in $(1/4,4)$ and  $\widehat \zeta(\rho)= 1$ on
$[1/2,2]$. Then $k_t=\zeta* h_t$ and an application of
\eqref{convwithSchwLor} shows that we can replace $h_t$ by $k_t$
in \eqref{ueven}.

We proceed to show \eqref{ueven}. 
First observe that $\phi m(t\cdot)$ is an $L^2$ function;
namely by assumption (iii)  and the Hausdorff-Young inequality it belongs to the dual space of $L^{q,\sigma}(\mu_d)$ and in view of its support to 
$L^2$. 
Since $\cB_1$ is the cosine
transform and since $\cB_d^2$ is the identity on $L^2$ functions the inequality \eqref{ueven} follows
from 
 \Be\label{transplt}
\big\|(1+(\cdot))^{-\frac{d-1}{2}} \cB_1[\chi \cB_d g]
\big\|_{L^{q,\sigma}((1+r)^{d-1} dr)} \lc
\|g\|_{L^{q,\sigma}(\mu_d)}, \quad q\le 2, \Ee applied to $g= \cB_d[\phi
m(t\cdot)]$. Here the function $\chi$ is assumed to be smooth and
supported in $(1/4,4)$ and equal to one on the support of $\phi$.
This inequality is
related to and could be derived from
the more sophisticated transplantation theorems
of  Stempak \cite{stempak} and Nowak and Stempak \cite{nst}
on the composition of
 nonmodified Hankel transforms, but \eqref{transplt}
has an easy direct proof:
We first note that \eqref{transplt} follows by real interpolation
from the $L^q$ inequalities, i.e. the case $q=\sigma$. Thus it
suffices to show \Be\label{transpltLq} \big\|\cB_1[\chi \cB_d g]
\big\|_{L^{q}((1+r)^{(d-1)(1-q/2)}dr)}\lc
\|g\|_{L^{q }(\mu_d)}. \Ee This in turn follows easily from
Corollary \ref{babytranspl} and an estimate of Hardy type. Indeed
changing variables $s=r+u$ and an application of Minkowski's
inequality yields \begin{multline*}\big\|\cB_1[\chi \cB_d g]
\big\|_{L^{q}((1+r)^{(d-1)(1-q/2)}dr)}
\\
\lc \int_{-\infty}^\infty (1+|u|)^{-N} \Big(\int_{r=-u}^\infty
(1+r)^{(d-1)(1-q/2)} \frac{|f(r+u)|^q (r+u)^{(d-1)q}}
{(1+|r+u|)^{\frac{d-1}{2}q}} dr \Big)^{1/q} du.
\end{multline*}
We use the estimate $(1+r)^\alpha\lc (1+|r+u|)^{\alpha}
(1+|u|)^{\alpha}$ for $\alpha=(d-1)(1-q/2)$. Thus the last
displayed term is seen to be bounded by
\[
C\int_{-\infty}^\infty (1+|u|)^{-N+(d-1)(1-q/2)}
\Big(\int_{-u}^\infty |f(r+u)|^q (r+u)^{d-1} dr \Big)^{1/q} du
\]
which for large $N$ is $\lc\|f\|_{L^q(\mu_d)}.$
This shows \eqref{transpltLq} and finishes the proof of the
implication $(iii)\implies (iv)$.\end{proof}

\section{Sufficiency: The basic decomposition}
\label{sufficiency}

In this section we begin the proof of  the main implication
\rm(iv)$\implies$(i) of Theorem \ref{mainhankel}.
Let $\vphi\in C^\infty_c(\frac 12, 2)$ as in \eqref{speccutoff}.
Let $\ka_j(r) =\cF_\bbR^{-1}[\vphi m(2^j\cdot)]$,
let \Be \label{Ajdef}
A_j(q,\sigma)=
\big\| (1+|\cdot|)^{-\frac{d-1}2} \ka_j\big\|_{L^{q,\sigma}(\nu)}
\Ee
with $d\nu=(1+|x|)^{d-1} dx$,
and
 \Be\label{defofA}A\equiv A(p,q,\sigma):=
\sup_j 2^{jd(\frac 1p-\frac 1q)} A_j(q,\sigma).\Ee

Define \Be \label{Kjdef} K_j = \cK_{d,d}[\vphi m(2^j\cdot)]\Ee
({\it cf.} \eqref{kernelintegral}) and
\Be \label{Tjdef}T^j f(r)=
\int 2^{jd}
K_j(2^jr, 2^j s) f(s) s^{d-1} ds.\Ee

 Define Littlewood-Paley cutoffs $L_j$, $\widetilde L_j$ by
$\cB_d[L_j f](\rho)= \vphi(2^{-j}\rho)\cB_d f(\rho)$ and
$\cB_d[\widetilde L_j f](\rho)= \eta(2^{-j}\rho)\cB_d f(\rho)$
where $\eta $  is supported in $(1/4,4)$ and equal to $1$ on the support of $\vphi$.
Then $\cB_d[m\cB_d f]=\sum_j L_j T^j \widetilde L_j f$.
We apply (the Lorentz space analogues of) the
 Littlewood-Paley inequalities  \eqref{LP1}, \eqref{LP2}
(one with the $L_j$, the other one with the $\widetilde L_j$).
Using also Lemma  \ref{indepofphi} (which justifies
the use of the specific cutoff function $\varphi$ in \eqref{speccutoff})
 we see that Theorem \ref{mainhankel} follows from
the inequalities for
vector-valued functions $\{f_j\}_{j\in \bbZ}$,
\Be\label{vectval}
\Big\|\Big(\sum_j|T^j f_j|^2\Big)^{1/2}\Big\|_{L^{q,\sigma}(\mu_d)}
\lc A(p,q,\sigma)
\Big\|\Big(\sum_j|f_j|^2\Big)^{1/2}\Big\|_{L^{p,\om}(\mu_d)}.
\Ee

For a further decomposition we introduce the notation
$$ \chi_n(r)= \chi_{[2^n, 2^{n+1}]}(r)$$
and decompose a.e. into three parts
\Be\label{threeparts}
T^j f= \sum_{n\in \bbZ}\chi_n \Big(
\sum_{m<j+n-5}
+ \sum_{\substack{j+n-5\le m\\ \le j+n+5}}
+\sum_{m>j+n+5} T^j[ f\chi_{m-j}]\Big).
\Ee

The first term will contribute to a Hardy type (or Hilbert integral type)
operator whose estimate
needs the full strength of the assumption.
The second term will contribute to a
singular integral operator, for  vector-valued functions,
whose estimation however will not require
the full strength of our assumption.
We consider the third term as an ``error'' term which
contributes again to a better behaved Hardy type operator.

We let
\begin{align}\label{opdefsH}
H_{j,m}f &= \sum_{n>m-j+5}\chi_n T^j[\chi_{m-j}f],
\\ \label{opdefsS}
S_{j,n,i}f&=\chi_n T^j[\chi_{n+i}f],
\\ \label{opdefsE}
E_{j,m} f &= \sum_{n<m-j-5}\chi_n T^j[\chi_{m-j}f].
\end{align}

By \eqref{threeparts}
$$T^j= \sum_{m\in \bbZ} H_{j,m}+ \sum_{n\in \bbZ}
\sum_{i=-5}^5 S_{j,n,i} +\sum_{m\in \bbZ} E_{j,m} .$$

We now state the main estimates regarding these three terms.
The implicit constants may depend on the parameters $p,q,\sigma,\eps,d$.
For the main term
we have

\begin{proposition} \label{Hprop}
For $m\in \bbZ$,
$1<p\le q<2$, $1\le \sigma\le \infty$
\begin{multline}\label{Hest}
\big\|H_{j,m} f\|_{L^{q,\sigma}(\mu_d)} \\ \lc
\min
\{ 2^{-m (d(\frac 1p-\frac 12)-\frac 12)},2^{m\frac{d}{p'}}\}
2^{jd(\frac1p-\frac 1q)}A_j(q,\sigma) \|f\|_{L^{p,\infty}} (\mu_d).
\end{multline}
\end{proposition}
Note that in the range of interest,
$1<p<\frac{2d}{d+1}$, these
 estimates can be summed in $m$.

The estimation of the remaining two terms \eqref{opdefsS}, \eqref{opdefsE}
does  not need the full strength of our
 assumptions. To formulate the appropriate weaker hypotheses
let, for
$\eps\ge 0$, $1\le u<2$
\begin{align}\label{defofBj} B_j(\eps,u)&=
\Big(\int_{-\infty}^\infty|\kappa_j(x)|^u(1+|x|)^{u\eps} dx\Big)^{1/u},
\\ \label{defofB}
B(\eps,p,q)&=\sup_j 2^{jd(1/p-1/q)} B_j(\eps,u(p,q)), \text{ where }
\frac{1}{u(p,q)}=\frac{\frac 1p+\frac 1q-1}{\frac 2p -1}.
\end{align}

\begin{proposition} \label{Eprop}
Let $\eps>0$, $1< p\le q< 2$, $1\le \sigma\le \infty$, and let
$\theta\equiv \theta(p,q)=(\frac 1p-\frac 1q)/(\frac 1p-\frac 12)$.
For $m\in \bbZ$,
\begin{multline}\label{Eest}
\big\|E_{j,m} f\|_{L^{q,\sigma}(\mu_d)}
\\ \lc B(4\eps(1-\theta),p,q) \min
\{ 2^{-m(1-\theta)\eps},2^{m(1-\theta)(d-1)}\}
\|f\|_{L^{p,\sigma}(\mu_d)}.
\end{multline}
\end{proposition}

The square-function estimates  associated to $\{S_{j,n,i}\}_{j\in \bbZ}$
can be seen  as estimates for vector-valued singular integrals under the assumption $B(\eps,p,q)<\infty$, for small $\eps>0$.

\begin{proposition} \label{Sprop}
For $n\in \bbZ$,  $-5\le i\le 5$,
$1<p< 2$,
\Be\label{Sest}
\Big\|\Big(\sum_j|S_{j,n,i} f_j|^2\Big)^{1/2}\Big\|_{L^{q,\sigma}(\mu_d)}
\lc B(\eps, p,q)
\Big\|\Big(\sum_j|f_j|^2\Big)^{1/2}\Big\|_{L^{p,\sigma}(\mu_d)}
\Ee
\end{proposition}

\medskip
To see that the conditions of Propositions \ref{Eprop} and \ref{Sprop}
are less restrictive than the condition \eqref{global} we note
\begin{lemma}\label{AversB}
Suppose $p< \frac{2d}{d+1}$, $p\le q<2$, and
$\frac 1u=\frac{p^{-1}+q^{-1}-1}{2p^{-1}-1}$. Then there is $\eps=\eps(p,q)>0$ so that
$B_j(\eps,u)\lc A_j(q,\sigma)$, for all $\sigma\le \infty$.
\end{lemma}
\begin{proof}
We begin by observing that
$(1+|x|)^{-\alpha}$ belongs to the Lorentz space $L^{\rho,1}(\nu)$
if and only if $\alpha\rho>d$.
Now  write
\[
B_j(\eps,u)= \Big(\int \frac{|\kappa(x)|^u}{(1+|x|)^{u\frac{d-1}{2}}}
(1+|x|)^{\eps u+
\frac{d-1}{2} u +1-d}
 d\nu(x)\Big)^{1/u}
\]
with $d\nu(x)= (1+|x|)^{d-1}$.
Note that by assumption  the $L^{q/u,\infty}(\nu)$ norm of
$|\kappa_j|^u(1+|x|)^{-u(d-1)/2}$ is bounded by
$A_j(q,\infty)^u$. Thus it suffices to check that for sufficiently small
$\eps$ the function
$$V_\eps(x)=(1+|x|)^{\eps u+
\frac{d-1}{2} u +1-d} $$
belongs to $ L^{(q/u)',1}(\nu)$.
This holds under the condition $(d-1)(1-u/2)>d(1-u/q)$.
Since
$ u^{-1}=\frac{p^{-1}+q^{-1}-1}{2p^{-1}-1}$
a straightforward computation shows that the condition is equivalent
to an inequality
which is
independent of $q\in [p,2)$, namely just
$p<2d/(d+1)$.
\end{proof}

For later use let us also  observe that $B(\eps)\equiv B(\eps,p,p)$ is
independent of $p$, namely
\Be \label{Beps}
B(\eps)= \sup_j \|\ka_j\|_{L^1((1+|x|)^\eps dx)}.\Ee
Moreover for   some real interpolations, we shall  need the
following locally uniform control on the constants $B(\eps,p,q)$.

\begin{lemma} \label{Brevinclusion}
Let $1<p\leq q<2$ and $\eps>0$. Then there
exist constants $C,\eta>0$ (depending on $\eps,p,q$) so that for all
$\tp\in(p-\eta,p+\eta)$ and $\tq$ satisfying
$\frac1{\tp}-\tfrac1{\tq}=\frac1p-\frac1q$ we have \[
B(\eps/2,\tp,\tq)\leq \,C\,B(\eps,p,q).
\]
\end{lemma}

\begin{proof} We first observe that when $u_1\ge
u$\Be\label{Bjinclusion} B_j(\eps, u_1) \leq\,C_\eps\,B_j(\eps,
u).\Ee Indeed, this follows from the fact that the Fourier
transform of $\kappa_j$ is compactly supported and therefore can
be written as a convolution
 with a Schwartz function; we then apply Lemma \ref{straightfconv}.

On the other hand, if $u_1< u$, by H\"older's inequality we have
\Be\label{Bjinclusion_reverse} B_j(\eps/2, u_1)
\leq\,C_{\eps,u,u_1}\,B_j(\eps, u),\Ee provided we choose
$\frac1{u_1}<\frac1u+\frac\eps2$. Now let $p-\eta<\tp<p+\eta$ and
define $\tq$ so that $\tp^{-1}-\tq^{-1}=p^{-1}-q^{-1}$, where
$\eta=\eta(\eps,p,q)>0$ is chosen so that
$|u(\tp,\tq)^{-1}-u(p,q)^{-1}|<\eps/4$. Then using either
\eqref{Bjinclusion} or \eqref{Bjinclusion_reverse} and
$\tp^{-1}-\tq^{-1}=p^{-1}-q^{-1}$ 
the asserted estimate follows. 
\end{proof}


\noi
{\bf Proof of
Theorem \ref{mainhankel},  given Propositions
\ref{Hprop}, \ref{Eprop}, \ref{Sprop}.}
We need to estimate  the square-function on the left hand side
of \eqref{vectval}  with $T^j$ replaced by one of the terms
$\sum_m H_{j,m}$,
$\sum_{n\in\bbZ}\sum_{i=-5}^5 S_{j,n,i}$, and
$\sum_m E_{j,m}$.

Observe that $H_{j,m} f_j = H_{j,m} [f_j\chi_{m-j}] $ and we bound
\begin{align*}
&\Big\|\Big( \sum_j\Big|\sum_m H_{j,m}f_j
\Big|^2\Big)^{1/2}\Big\|_{L^{q,\sigma}(\mu_d)}
\le
\sum_{m}\Big\|\Big( \sum_j\big|H_{j,m}[f_j\chi_{m-j}]\big|^2\Big)^{1/2}
\Big\|_{L^{q,\sigma}(\mu_d)}
\\
&\le
\sum_{m}\Big(\sum_j\big\|H_{j,m}[f_j\chi_{m-j}]
\big\|_{L^{q,\sigma}(\mu_d)}^\om\Big)^{1/\om},\quad \om=\min\{q,\sigma\}.
\end{align*}
Here we have used Minkowski's inequality for the $m$-summation,
followed by Lemma \ref{vectorlor}.
Let $\delta(p)=\min\{d/p', d(1/p-1/2)-1/2\}$ then $\delta(p)>0$ for $1<p<\frac{2d}{d+1}$ and by Proposition \ref{Hprop}
the last expression in the displayed formula is bounded by
$C A(p,q,\sigma)$ times
\begin{align*}
&\sum_{m\in \bbZ} 2^{-|m|\delta(p)}
\Big(\sum_j\big\|f_j\chi_{m-j}\big\|_{L^{p,\infty}(\mu_d)}^\om\Big)^{1/\om}
\\&
\lc\sum_{m\in \bbZ} 2^{-|m|\delta(p)}
\Big(\sum_j\big\|f_j\chi_{m-j}\big\|_{L^{p,\om}(\mu_d)}^\om\Big)^{1/\om}
\\&\lc\sum_{m\in \bbZ} 2^{-|m|\delta(p)}
\Big\|\sup_{j}|f_j\chi_{m-j}|\Big\|_{L^{p,\om}(\mu_d)}
\\&\lc \Big\|\sup_j |f_j|\Big\|_{L^{p,\om}(\mu_d)}
\lc  \Big\|\Big(\sum_j |f_j|^2\Big)^{1/2}\Big\|_{L^{p,\om}(\mu_d)} .
\end{align*}
Here, in order to bound the second expression, we have used
\eqref{quasinormlor}, and the assumption that $\om\ge p$, together with the
disjointness of the intervals $[2^{m-j}, 2^{m-j+1})$.
This  completes the proof of the $L^{p,\om}(\ell^2, \mu_d)\to
L^{q,\sigma}(\ell^2, \mu_d)$ bound for
$\{\sum_m H_{j,m}f_j\}_{j\in \bbZ}$.
The  terms $\{\sum_m E_{j,m}f_j\}_{j\in \bbZ}$ are estimated similarly,
given Proposition \ref{Eprop} and
Lemma \ref{AversB}.

Concerning the
terms $S_{j,n,i}$, let us consider the $L^p\to L^q$ estimates.
We recall
$S_{j,n,i}f_j
 = \chi_n S_{j,n,i} [f_j\chi_{n+i}]$ and use Proposition \ref{Sprop}, for fixed $i$, and $n$.
In view of the cutoffs $\chi_n(r)$, $\chi_{n+i}(s)$, $-5\le i\le 5$
the uniform Lebesgue space  estimate of Proposition
\ref{Sprop} also gives an $L^p(\mu_d)$ estimate for the sum,
\[
\Big\|\Big(\sum_j
\Big|\sum_{n} S_{j,n,i} f_j\Big|^2\Big)^{1/2}
\Big\|_{L^q(\mu_d)} \le C_{\eps,p} B(\eps,p,q)
\Big\|\Big(\sum_j|f_j|^2\Big)^{1/2}
\Big\|_{L^p(\mu_d)}.
\]
We sum in $i\in\{-5,\dots, 5\}$ and by
Lemma \ref{AversB} we obtain the desired $L^p\to L^q$
estimate for the
singular integral part  in the range $1<p<2d/(d+1)$.
By real interpolation (and Lemma \ref{Brevinclusion}) this extends to the
$L^{p,\sigma}\to L^{q,\sigma}$  estimates.



\section{ Proof of  Proposition \ref{Hprop}}
\label{Hardybounds}
Let $I_n=[2^n, 2^{n+1}]$,  and $ \cR_{n}=[2^n,\infty)$.
We estimate
\begin{align}
&\Big\|\sum_{n>m-j+5} \chi_n T^j[f\chi_{m-j}]\Big\|_{L^{q,\sigma}(\mu_d)}\notag
\\&\le
\Big\| \chi_{\cR_{m-j+5}} \int 2^{jd} |\cK_j(2^j\cdot, 2^js)| |f(s)|
\chi_{m-j}(s) s^{d-1} ds \Big\|_{L^{q,\sigma}(\mu_d)}
\notag
\\&= 2^{-jd/q}
\Big\| \chi_{\cR_{m+5}} \int_{I_m}  |\cK_j(\cdot, s)| |f(2^{-j}s)|
 s^{d-1} ds \Big\|_{L^{q,\sigma}(\mu_d)}
\label{afterscaling}
\end{align}
by changes of variables in $s$ and $r$.

We now use the kernel estimate of Proposition
\ref{kernest} and set
\Be\label{defofWj} W_j(x)= \int \frac{|\ka_j(x-u)|}{(1+|u|)^N} du.\Ee
We apply Minkowski's inequality
(i.e. the continuous form of the triangle inequality in the Lorentz space
$L^{q,\sigma}$ which is a Banach space) and see that the expression
\eqref{afterscaling} is controlled by
\[
2^{-jd/q}
\int_{I_m}|f(2^{-j}s)| \frac{s^{d-1}}{(1+s)^{\frac{d-1}{2}}}
 \sum_{(\pm,\pm)}
\Big\| \chi_{\cR_{m+5}}
\frac{W_j(\pm \cdot \pm s)}{(1+\cdot )^{\frac{d-1}{2}}} \Big\|_{L^{q,\sigma}(\mu_d)} ds.
\]
It is now crucial that in the inner norm the functions are restricted to
the set where $r\ge 2^{m+5}$
while $s\le 2^{m+1}$. We may therefore change variables and use the bound
$(1+|r-s|)\ge c (1+r)$ in this range, so that
\[
\Big\| \chi_{\cR_{m+5}}
\frac{W_j(\pm \cdot \pm s)}{(1+\cdot )^{\frac{d-1}{2}}} \Big\|_{L^{q,\sigma}(\mu_d)}
\lc \Big\|
\frac{W_j}{(1+|\cdot| )^{\frac{d-1}{2}}} \Big\|_{L^{q,\sigma}(\nu)},\quad
s\le2^{m+1},
\]
where
$d\nu= (1+|x|)^{d-1} dx$.
By Lemma
\ref{straightfconvLor}
the term on the right hand side is also controlled by
$\big\|\ka_j(1+|\cdot| )^{-\frac{d-1}{2}} \big\|_{L^{q,\sigma}(\nu)}$, which is
 $A_j(q,\sigma)$.

Thus we see that the expression \eqref{afterscaling} is bounded by
\[C 2^{j d(1/p-1/q)} A_j(q,\sigma)
\int_{I_m} 2^{-jd/p}|f(2^{-j}s) |
(1+s)^{-(d-1)/2} s^{d-1}ds.\]
It remains to bound the $s$-integral. It is easy to check that the restriction
of
$\Omega(s)= (1+s)^{-(d-1)/2}$ to the interval $I_m$ belongs to
$L^{p',1}(I_m, \mu_d)$ and satisfies the bounds
\[\big\|\chi_m \Omega\big\|_{L^{p',1}(\mu_d)}\lc
\begin{cases} 2^{-m (d(1/p-1/2)-1/2)}
&\text{ if } m\ge 0,
\\ 2^{md/p'}&\text{ if } m\le 0,
\end{cases}
\]
and thus, by duality
\begin{align*}
&\int_{I_m} 2^{-jd/p}|f(2^{-j}s) |  \frac{s^{d-1}}{(1+s)^{\frac{d-1}2}} ds
\le \|\chi_m \Omega \|_{L^{p',1}(\mu_d)} \|2^{-jd/p}f(2^{-j}\cdot)
\|_{L^{p,\infty}(\mu_d)}
\\
&\lc \min\{ 2^{-m (d(1/p-1/2)-1/2)}, 2^{md/p'}\} \|f\|_{L^{p,\infty}(\mu_d)}.
\end{align*}
This finishes the proof.\qed
\section{More $L^p$ estimates}
\label{sing-hardy}

In this section we consider the case $p=q$ of Propositions
\ref{Eprop} and \ref{Sprop}; the  general case will be handled
in \S\ref{interpolsect}. The results of this section together with the previous section complete the proof of Theorem \ref{mainhankel} in the case $p=q$.
In what follows we shall  assume $p=\sigma$
in the proof of Proposition \ref{Sprop}
 since the $L^{p,\sigma}$ boundedness results follow then  by interpolation and replacing
$\eps$ with $\eps/2$. Moreover we prove the statement of Proposition 
\ref{Eprop} for the case 
$p=q=\sigma$
with the constant
with the constant $B(\eps)$ (rather than $B(4\eps)$), and  the factor $4$
is included in the statement of the proposition 
to account 
for  interpolations needed for  the general case (\cf. also Lemma 
\ref{Brevinclusion}).

\begin{proof} [\bf Proof of  Proposition
\ref{Eprop}, $p=q=\sigma$]
We begin  with the estimate \eqref{afterscaling}
which is still valid but continue differently since now $n+j\le m-5$, thus $r\ll s$. Let $I_m^*=[2^{m-1}, 2^{m+2}]$.
Set $ h_{p,j}(s) = 2^{-jd/p} f(2^{-j }s) s^{\frac{d-1}{p}}.$ Then the right hand side of \eqref{afterscaling} is estimated by
\begin{align*}
&\sum_{(\pm,\pm)}
\Big(\sum_{n\le m-j-5}\int_{I_{n+j}}\Big|\int_{I_m}
\frac{ |W_j(\pm r\pm s)|}{(1+s)^{\frac{d-1}{2}}}
\frac{h_{p,j}(s)s^{(d-1)/p'}} {(1+r)^{\frac{d-1}{2}}}
ds \Big|^p r^{d-1}dr\Big)^{1/p}
\\
&\lc\sum_{(\pm,\pm)}
\Big(\int_0^{2^{m-3}}\Big|\int_{I_m^*}
\frac{ |W_j(\pm y)|}{(1+y)^{\frac{d-1}{2}}}
[\chi_m h_{p,j}](y\pm r)
dy \Big|^p \frac{2^{m(d-1)\frac {p}{p'}}    r^{d-1}}
{(1+r)^{\frac{d-1}{2}p}}dr\Big)^{1/p}.
\end{align*}
If $m>0$ this is dominated by
\begin{align*}&C\sum_{\pm}
 2^{-m\eps}
\int|W_j(y)|(1+|y|)^\eps
\Big(\int\big|[\chi_m h_{p,j}](y\pm r) \big|^p dr\Big)^{1/p}dy
\\&\lc  2^{-m\eps}
\|\ka_j\|_{L^1((1+|\cdot|)^\eps dy)}
\|f\chi_{m-j}\|_{L^p(\mu_d)}.
\end{align*}
If $m<0$ we may instead estimate $2^{m(d-1)p/p'}r^{d-1}\le 2^{m(d-1)p}$; this yields the bound
\[2^{m(d-1)}
\|\ka_j\|_1
\|f\chi_{m-j}\|_{L^p(\mu_d)}\]
instead. This finishes the proof.
\end{proof}

\begin{proof} [\bf Proof of  Proposition
\ref{Sprop}, $p=q=\sigma$]
We use standard arguments for singular integrals for $\ell^2$-valued
kernels and functions. First, by orthogonality,
\begin{align*}
&\Big\|\Big(\sum_j|S_{j,n,i} f_j|^2\Big)^{1/2}
\Big\|_{L^2(\mu_d)}
\le \Big\|\Big(\sum_j\big|T^j[ f_j\chi_{n-i}]
\big|^2\Big)^{1/2}\Big\|_{L^2(\mu_d)}
\\
&\lc \sup \|\widehat \kappa_j\|_\infty
\Big\|\Big(\sum_j| f_j|^2\Big)^{1/2}\Big\|_{L^2(\mu_d)}.
\end{align*}
To prove the $L^p(\mu_d)$ bounds for $1<p<2$ it suffices,
by the Marcinkiewicz interpolation theorem,  to prove
the weak type $(1,1)$ inequality
\Be\label{weaktype}\mu_d\Big(\Big\{r:
\Big(\sum_j|S_{j,n,i} f_j|^2\Big)^{1/2}>\la\Big\}\Big)\lc B\la^{-1}
\Big\|\Big(\sum_j|f_j|^2\Big)^{1/2}\Big\|_{L^1(\mu_d)};
\Ee
here $B=B(\eps)$ as in \eqref{Beps}.

Set $h_j(s)=f_j(s) (2^{-n} s)^{d-1}\chi_{n+i}(s)$, so that $|h_j|$ and $|f_j|$ are of comparable size on $I_{n+i}$.
For fixed $\la>0$ we make a Calder\'on-Zygmund decomposition of the $\ell^2$ valued function $\{h_j\}$, at height
$\la/B$  (see \cite{stein-si}). We thus decompose  $h_j= g_j+b_j$
where  $\|\vec g\|_{L^\infty(\ell^2)} \le \la/B$,
  $\|\vec g\|_{L^1(\ell^2, ds)} +
\|\vec b\|_{L^1(\ell^2,ds)} \lc
\|\vec h\|_{L^1(\ell^2, ds)} $.
Furthermore  $b_j= \sum_\nu
b_{j,\nu}$ so that $b_{j, \nu} $ is supported in a dyadic subinterval
$J_\nu$
 of $I_{n+i}$, with center $s_\nu$ and length $2^{L_\nu}$. The interiors of the intervals  $J_\nu$ are disjoint, and we have
$|J_\nu|^{-1}\int_{J_{\nu}}|\vec b_{\nu}(s)|_{\ell^2}ds\lc \la/B$
and
$\sum_\nu |J_\nu|\le B\la^{-1}\|\vec h\|_{L^1(\ell^2,ds)}$.
Finally $\int b_{j,\nu}ds=0$ for all $j,\nu$.

Note that $S_{j,n,i} f_j=
S_{j,n,i} \vec \gmod_j+ \sum_\nu S_{j,n,i}
\vec \bmod_{j,\nu}$ where
$\gmod_j(s)=g_j(s)(2^n/s)^{d-1}$ and
$\bmod_{j,\nu}(s)= b_{j,\nu}(s) (2^n/s)^{d-1}$.
We estimate
\begin{align}
&\mu_d(\{r\in I_{n+i}: |\{S_{j,n,i} \gmod_j(r)\}|_{\ell^2} >\la/2\})
\lc \la^{-2} B^2 \|\vec \gmod\|_{L^2(\ell^2,\mu_d)}^2
\notag
\\&\lc \la^{-1} B \|\vec \gmod\|_{L^1(\ell^2,\mu_d)}
\lc \la^{-1} B \|\vec f\|_{L^1(\ell^2,\mu_d)}.
\label{goodfunctionest}
\end{align}

For each interval $J_\nu$ let $J_\nu^*$ denote the interval with
same center and tenfold length. Also let $\Omega=\cup_\nu J_{\nu}^*$ then
\Be\label{measexcset}
\mu_d(\Omega) \lc 2^{n(d-1)}\sum|J_\nu|
\lc B\la^{-1}2^{n(d-1)}\|\vec h\|_{L^1(\ell^2,ds)}
\lc B\la^{-1}\|\vec h\|_{L^1(\ell^2,\mu_d)}.
\Ee

It remains to estimate

\begin{align}\label{offexcset}
&\mu_d\Big(\Big\{r\in I_n\setminus\Omega:
\Big(\sum_j\big|S_{j,n,i} \big[\sum_\nu \bmod_{j,\nu}\big]
\big|^2\Big)^{1/2}>\la/2\Big\}\Big)
\\
\notag
&\lc \la^{-1} \int_{I_n\setminus \Omega} \Big(\sum_j\Big|S_{j,n,i}
 \big[\sum_\nu
\bmod_{j,\nu}\big]\Big|^2\Big)^{1/2} r^{d-1} dr
\\
\label{badfunctionbd}
&\lc
\la^{-1} 2^{n(d-1)} \sum_\nu\sum_j \int_{I_n\setminus J_\nu^*}
\big|S_{j,n,i}\bmod_{j,\nu}(r)\big| dr.
\end{align}
Note that
\[S_{j,n,i}\bmod_{j,\nu}(r)= 2^{n(d-1)}\int 2^{jd} \cK_j(2^j r, 2^j s) b_{j,\nu}(s) ds.\]
To estimate the integral in \eqref{badfunctionbd}
we distinguish the cases $j\ge -L_\nu$,
$j\le -L_\nu$. Note that
 $L_\nu\le n+5$ as $J_\nu\subset I_{n+i}$.

If $j\ge -L_{\nu}$ ($\ge -n-5$) we use the kernel estimate of Proposition
\ref{kernest} and obtain, with  the notation  $W_j$ in \eqref{defofWj} and $r,s\approx 2^n$
\begin{align*}
|S_{j,n,i}\bmod_{j,\nu}(r)|&\lc \sum_{\pm,\pm}
\frac{2^{jd}2^{n(d-1)}}{(1+2^jr)^{\frac{d-1}{2}} (1+2^js)^{\frac{d-1}{2}}}
\int W_j(\pm 2^jr\pm 2^j s)
|b_{j,\nu}(s)| ds
\\&\lc \sum_{\pm,\pm} \int 2^{j} W_j(\pm 2^jr\pm 2^j s)
|b_{j,\nu}(s)| ds
\end{align*}
and if $r\notin J_\nu^*$ then $|r-s|\approx |r-s_\nu|>2^{L_\nu}$.
Consequently
\begin{align}\notag \int_{I_n\setminus J_\nu^*}
\big|S_{j,n,i}\bmod_{j,\nu}(r)\big| dr
&\lc \int_{|x|>2^{j+L_\nu}}|W_j|(x) dx \,\int|b_{j,\nu}(s)|ds
\\ \label{jlarge} &\lc 2^{-(j+L_\nu)\eps}B(\eps)
\big\| b_{j,\nu}\big\|_{L^1(ds)}.
\end{align}
If $j<-L_\nu$ we use the cancellation of the $b_{j,\nu}$
to write
\begin{align*}
\big|S_{j,n,i}\bmod_{j,\nu}(r)\big|
= 2^{n(d-1)}
\Big|\int 2^{jd} \big[
\cK_j(2^j r, 2^j s) -\cK_j(2^j r, 2^j s_\nu) \big]
b_{j,\nu}(s) ds\Big|&
\\ \lc 2^{n(d-1)}2^{j+L_\nu} \int_{\sigma=0}^1\int 2^{jd} \big|
\partial_s\cK_j(2^j r, 2^j(s_\nu+\sigma(s-s_\nu)))\big|
\, \big|b_{j,\nu}(s)|ds&.
\end{align*}
We now argue as before, but use Proposition \ref{kernest} to estimate
$\partial_s\cK_j$ and we obtain for $j\le -L_\nu$
\begin{align}
\notag
&\int_{I_n\setminus J_\nu^*}
\big|S_{j,n,i}\bmod_{j,\nu}(r)\big| dr
\\
\notag& \lc
\int\big|b_{j,\nu}(s)|ds\,2^{j+L_\nu} 2^{n(d-1)}\sum_{(\pm,\pm)}
\sup_a\int \frac{2^{jd} W_j(\pm2^j r\pm2^j a)}
{1+ 2^{(j+n)(d-1)}} dr
\\
\notag& \lc
\int\big|b_{j,\nu}(s)|ds\,2^{j+L_\nu} \sum_{(\pm,\pm)}
\sup_a\int 2^j W_j(\pm2^j r\pm2^j a) dr
\\& \label{jsmall} \lc B(0) 2^{j+L_\nu}\|b_{j,\nu} \|_{L^1(ds)}.
\end{align}

We can sum the  terms  \eqref{jlarge} and \eqref{jsmall}
in $j$ and obtain
\begin{align*}
&\sum_j\int_{I_n\setminus J_\nu^*}
\big|S_{j,n,i}\bmod_{j,\nu}(r)\big| dr
\\&\lc B(\eps) \sum_{j} \min\{  2^{j+L_\nu}, 2^{-(j+L_\nu)\eps}\}
\big\| b_{j,\nu}\big\|_{L^1(ds)}
\lc B(\eps)  \big\| \vec b_{\nu}\big\|_{L^1(\ell^2,ds)}.
\end{align*}
Now we sum in $\nu$ and get the required $L^1(\mu_d)$ bound off $\Omega$.
The expression \eqref{badfunctionbd} is thus dominated by
\begin{align*}
&\la^{-1}B(\eps) \sum_\nu 2^{n(d-1)}\big\| \vec b_{\nu}\big\|_{L^1(\ell^2,ds)}
\\&\lc \sum_\nu |J_\nu|2^{n(d-1)} \lc
\la^{-1}B(\eps) 2^{n(d-1)} \int_{I_n}|\vec h(s) |_{\ell^2}ds
\\&\lc
\la^{-1}B(\eps)  \int_{I_n}|\vec h(s) |_{\ell^2}s^{d-1}ds.
\end{align*}
This bounds the expression \eqref{offexcset} by
$C B(\eps)\la^{-1}
\| \vec f\|_{L^1(\ell^2,\mu_d)}$. Combining this bound with
\eqref{goodfunctionest} and \eqref{measexcset} yields
the desired  weak type $(1,1)$ bound \eqref{weaktype}.
\end{proof}

\section{$L^p\to L^2$ estimates}\label{Lp2est}
In this section we prove some sharp $L^p\to L^2$ bounds for Hankel multipliers.

\begin{theorem}
\label{Lp2thm}
Let $d>1$.

(i) Suppose  $1<p<\frac{2d}{d+1}$. Then  $m\in \fM_d^{p,2}$ if and only if
\Be \label{LpL2cond}
\sup_{t>0} \, t^{d(\frac 1p-\frac 12)}
\Big(\int_t^{2t}|m(\rho)|^2 \frac{d\rho}{\rho}\Big)^{1/2}<\infty.
\Ee

(ii) Let $p_d=\frac{2d}{d+1}$. Then the operator  $T: f\mapsto
\cB_d[m\cB_d f]$ maps the Lorentz space
$L^{p_d,1}(\mu_d)$ to $L^2(\mu_d)$ if and only if
\eqref {LpL2cond} holds for $p=p_d$.
\end{theorem}

\noi{\it Remark.} It is easy to see that the condition \eqref{LpL2cond} is equivalent to
\Be \label{LpL2condmod}
\sup_{t>0}t^{d(\frac 1p-\frac 12)}\|\phi m(t\cdot)\|_2 <\infty
\Ee
for some nontrivial, smooth $\phi$ with compact support in $(0,\infty)$.
\begin{proof}[Proof of Theorem  \ref{Lp2thm}]
We first prove (i). The necessity of the condition has already
been established  in \S\ref{easy}. For the proof of the sufficiency let
 $T^j$ be as in \eqref{Tjdef}. We  then show the  estimate
\Be \label{fixedTjL2}\big\|T^j f\big\|_{L^2(\mu_d)} \lc A_j(p,2) \|f\|_{L^p(\mu_d)}
\Ee
where $A_j(p,2)= 2^{jd(\frac 1p-\frac 12)}\|\vphi m(2^j\cdot)\|_2 $.
Note that by Plancherel's theorem and the argument of Lemma
\ref{indepofphi} the condition
$\sup_j A_j(p,2)<\infty$ is equivalent with \eqref{LpL2cond}
(and also with \eqref{LpL2condmod}).
Now,
\begin{align*}
&\big\|T^j f\big\|_{L^2(\mu_d)}
= \Big(\int
\Big[\int 2^{jd} \cK_j(2^jr, 2^j s) f(s) s^{d-1} ds\Big]^2 r^{d-1}dr\Big)^{1/2}
\\
&= 2^{-jd/2}
\Big(\int
\Big[\int  \cK_j(r,  s) f(2^{-j}s) s^{d-1} ds\Big]^2r^{d-1} dr\Big)^{1/2}
\\
&\lc  2^{-jd/2} \sum_{(\pm,\pm)}\int_0^\infty |f(2^{-j}s)
|\frac{s^{d-1}}{(1+s)^{\frac{d-1}{2}}} \Big(\int_0^\infty
\Big|
\frac{W_j(\pm r\pm s)}{(1+r)^{\frac{d-1}{2}}}\Big|^2 r^{d-1}dr\Big)^{1/2}
\, ds
\end{align*}
where
for the last bound we used Minkowski's inequality and the kernel estimate from Proposition \ref{kernest}. The last expression is controlled by
\begin{align}\label{duality}
&2^{-j d/2}\|\ka_j\|_2 \int|f(2^{-j}s)| \frac{s^{d-1}}{(1+s)^{\frac{d-1}{2}}} ds
\\ \notag
&\le
2^{-j d/2}\|\ka_j\|_2
\Big(\int|f(2^{-j}s)|^p s^{d-1} ds\Big)^{1/p}
\Big(\int \frac{s^{d-1}}{(1+s)^{\frac{d-1}{2}p'}} ds\Big)^{1/p'}
\end{align} and the second  integral in the last line is finite for  $p<\frac{2d}{d+1}$.
Changing variables  we  obtain
\[\big\|T^j f\big\|_{L^2(\mu_d)}\lc
2^{jd(1/p- 1/2)}\|\ka_j\|_2  \|f\|_{L^p(\mu_d)}.\]

We now use orthogonality and Littlewood-Paley theory, writing
$L_j f= \cB_d[\chi(2^{-j} \cdot)\cB_d f]$ and $T^j= L_jT^j L_j$ to get
\begin{align*}\|Tf\|_{L^2(\mu_d)}
&\lc\Big(\sum_j\|T^j L_j f\|_{L^2(\mu_d)}^2\Big)^{1/2}
\\&\lc \sup_j 2^{jd(1/p- 1/2)}\|\ka_j\|_2
\Big(\sum_k\| L_k f\|_{L^p(\mu_d)}^2\Big)^{1/2}
\end{align*}
and the argument is concluded by observing that for $1<p\le 2$
\Be\label{LPapplication}
\Big(\sum_k\| L_k f\|_{L^p(\mu_d)}^2\Big)^{1/2}\le
\Big\|\Big(\sum_k|L_k f|^2\Big)^{1/2}\Big\|_{L^p(\mu_d)} \le C_p
\|f\|_{L^p(\mu_d)}.
\Ee

The proof of (ii) is largely analogous. We may assume that $f$ is the
characteristic function of a measurable set $E$. The difference is
the estimate \eqref{duality}. We now observe that the function
$\om_d(s)=(1+s)^{-\frac{d-1}{2}}$ belongs to the space
$L^{p_d',\infty}(\mu_d)$ and by the duality between $L^{p_d,1}$ and
$L^{p_d',\infty}$
we use instead
\[ \int|\chi_E(2^{-j}s)| \frac{s^{d-1}}{(1+s)^{\frac{d-1}{2}}} ds
\lc
\big\|\chi_E(2^{-j}\cdot)\big\|_{L^{p_d,1}(\mu_d)}
\big\|\om_d\big\|_{L^{p_d',\infty}(\mu_d)} \]
which is $\lc [2^{jd} \mu_d(E)]^{1/p}$.
The subsequent Littlewood-Paley argument is the same; we use $f=\chi_E$ in \eqref{LPapplication}.
\end{proof}

\noi{\it Sharpness.} The restricted strong type $(p_d,2)$-estimate is sharp,
as the
Lorentz space $L^{p_d,1}$ cannot be replaced by $L^{p_d,\sigma}$ for $\sigma>1$.
To see this let $m_N(\rho)=\sqrt N \chi_{[1,1+cN^{-1}]}$ so that the condition
\eqref{LpL2cond} is satisfied uniformly in $N$.
Let $f_N(s)= s^{-(d+1)/2} e^{-is}\chi_{[1,N]}(s)$. Then one computes that
\[\|f_N\|_{L^{p_d,\sigma}(\mu_d)} \lc (\log N)^{1/\sigma}\] and
using the asymptotic expansion \eqref{besselasympt} one computes that
\[
\cB_d f_N(\rho) = c\int_1^N e^{i(\rho-1) s} \frac{ds}s +O(1)
\]
for $\rho$ near $1$ (observe that the corresponding integral with phase $-(\rho+1)s$ is bounded near $\rho=1$, by an integration by parts). Thus $|\cB_d f_N(\rho)|\gc \log N$ for
$|\rho-1|\le cN^{-1}$ (if $c$ is sufficiently small). Consequently
\[\|\cB_d[m_N \cB_d f_N]\|_{L^2(\mu_d)}
\approx \|m_N \cB_d f_N\|_{2} \gc \log N\]
which implies the assertion.

\medskip

\noi{\it Analogue for radial Fourier multipliers.}
We also note that an analogue of Theorem \eqref{Lp2thm} holds for radial Fourier multipliers acting on general $L^p(\bbR^d)$ functions, namely there is the 
\lq folk' result
\begin{observation} Suppose that $1<p\le \frac{2(d+1)}{d+3}$.
Then the operator $f\mapsto \cF^{-1}[m(|\cdot|)\widehat f]$
extends to a bounded operator from $L^p(\bbR^d)$ to $L^2(\bbR^d)$ if
and only if \eqref{LpL2cond} holds.
\end{observation}
\begin{proof}
The necessity has been observed in \S\ref{easy}.
If $m_t$ is supported in $\{\xi:t\le |\xi| \le 2t\}$ then it
follows by a well known argument of Fefferman \cite{F} from the
Stein-Tomas restriction theorem
(\cite{steinharmonic}, ch.IX-2) that
\begin{align*}&\|\cF^{-1}[m_t(|\cdot|)\widehat f]\|_2
\lc \Big(\int_{t}^{2t} |m_t(r)|^2\int_{S^{d-1}}|\widehat f(r\xi')|^2 d\sigma(\xi') r^{d-1} dr\Big)^{1/2}
\\&\lc
\Big(\int_t^{2t} |m_t(r)|^2
\|\tfrac{1}{r^d}f( \tfrac{\cdot}{r})
\|_p^2 r^{d-1} dr\Big)^{\frac 12}
= \|f\|_p
\Big(\int_{t}^{2t} |m_t(r)|^2 r^{2(\frac dp-\frac d2)} \frac {dr}{r}
\Big)^{\frac 12}.
\end{align*}
For global multipliers the result follows now by Littlewood-Paley theory exactly as in the proof of
Theorem \ref{Lp2thm}.
\end{proof}
We note that the restriction $p\le \frac{2(d+1)}{d+3}$
for the result  on general $L^p$ functions is optimal
as follows from the usual Knapp counterexamples for the
restriction theorem.

\section{Conclusion of the proof}\label{interpolsect}
In order to finish the  proof of Theorem \ref{mainhankel}
it just remains  to establish the  $L^p(\mu_d)\to L^q(\mu_d)$ estimates
in Propositions \ref{Eprop} and \ref{Sprop} for $p<q<2$. The appropriate
$L^{p,\sigma}(\mu_d)\to L^{q,\sigma}(\mu_d)$ follow then
 by the real interpolation method, if we take into account
Lemma 
\ref{Brevinclusion}.

The interpolations follow
 results on bilinear interpolation with the complex methods
(i.e. in disguise versions of
Stein's interpolation theorem  for analytic families),
see Theorems 4.4.1 and 4.4.2 in \cite{BL}. Using the first
(and more elementary)  of these results
we interpolate the inequalities
\begin{align*}
&\big\|E_{j,m} f\|_{L^{p}(\mu_d)}\lc
\min
\{ 2^{-m\eps},2^{m(d-1)}\}
 \|\ka_j\|_{L^1((1+|x|)^\eps dx)}
\|f\|_{L^p(\mu_d)},
\\
&\big\|E_{j,m} f\|_{L^{2}(\mu_d)}\lc
2^{jd(1/p-1/2)}
 \|\ka_j\|_{L^2} \|f\|_{L^p(\mu_d)},
\end{align*}
where the first bound has been already been established in \S\ref{sing-hardy}
and
the second is immediate from
\eqref{fixedTjL2}.
Similarly for the singular integrals we interpolate
\begin{align*}
\Big\|\Big(\sum_j|S_{j,n,i} f_j|^2\Big)^{1/2}\Big\|_{L^{p}(\mu_d)}
&\lc \sup_j\|\ka_j\|_{L^1((1+|x|)^\eps dx)}
\Big\|\Big(\sum_j|f_j|^2\Big)^{1/2}\Big\|_{L^{p}(\mu_d)},
\\
\Big\|\Big(\sum_j|S_{j,n,i} f_j|^2\Big)^{1/2}\Big\|_{L^{2}(\mu_d)}
&\lc \sup_j2^{jd(\frac 1p-\frac 12)}\|\ka_j\|_{2}
\Big\|\Big(\sum_j|f_j|^2\Big)^{1/2}\Big\|_{L^{p}(\mu_d)},
\end{align*}
where again the first inequality
has been proved in \S\ref{sing-hardy} and
the second follows from
\eqref{fixedTjL2} and Minkowski's inequality.
In order to obtain
the interpolated $L^p(\mu_d)\to L^q(\mu_d)$ statements we use
Lemma \ref{interpolfloc},
and Theorem 4.4.2  in \cite{BL}
(which involves the $[\cdot,\cdot]^\vth$ functor on one of the entries).
The proof is complete.\qed

We remark that for the interpolation of the singular operators
one  could have also based  the proof on
the more elementary Theorem 4.4.1 in \cite{BL}
which only involves the $[\cdot,\cdot]_\vth$ method; one then has to use
 the fact
that the
 space of  $L^p(\mu_d)$ functions $f$ for which
 $\cB_d f $ has compact support in $(0,\infty)$ is dense in $L^p(\mu_d)$,
see \cite{st-tr}. Thus one can reduce matters to uniform estimates for
compactly supported multipliers and apply  the interpolation result on the spaces $LF_o(p,a,b)$ mentioned in the remark following Lemma
\ref{interpolfloc}.

\section{Miscellanea}
\label{compmult}

\begin{proof}[{\bf Proof of Corollary \ref{besovcor}}]
The
$L^q((1+|r|)^{(d-1)(1-q/2)}dr)$ norm of a function  $\ka$ is dominated
using H\"older's inequality by
\[\Big(\sum_{j=0}^\infty \|\ka\|_{L^q(\cI_j)}^q
2^{j(d-1)(\frac 1q-\frac 12)q} \Big)^{1/q}
\lc \Big(\sum_{j=0}^\infty \|\ka\|_{L^2(\cI_j)}^q
2^{j d(\frac 1q-\frac 12)q} \Big)^{1/q}.
\]
This is applied to
$\ka=\cF^{-1}[\phi m(t\cdot)]$ and
the result follows from the definition of the Besov space.
\end{proof}

\begin{proof}[{\bf Proof of Corollary \ref{interpolcor}}]
This is an immediate consequence of theorem \ref{mainhankel} and the
interpolation formula of Lemma \ref{interpolfloc},
with varying $a$, $b$ (we set $a_i= (d_i-1)(1/q_i-1/2)$ and
 $b_i= d(1/p_i-1/q_i)$ for $i=0,1$).
\end{proof}

\noi{\bf Real interpolation.}
We can also prove some interpolation results using the real method, in view of
the nature of our conditions these are limited to the $K_{\vth,\infty}$ method with a number of restrictions (see \cite{BL} for general references about real interpolation).

Define
$\fM_{d}^{p,q,\sigma}$
as  the space  of all locally  integrable
functions $m$ on $\bbR_+$ for which $T_m$ extends to a
bounded operator from
$L^p(\mu_d)$ to $L^{q,\sigma}(\mu_d)$;  the norm is given by the
operator norm of $T_m$. Thus  $\fM_{d}^{p,q} =\fM_{d}^{p,q,q}.$

Theorem \ref{main} is used to prove
that for fixed $d$ the  weak type multiplier spaces
$\fM^{p,p}_{d,\infty}$, $1<p<2d/(d+1)$,  are stable under real interpolation,
with respect to the
$K_{\vth,\infty}$ method.

\begin{corollary} \label{interpolcorreal}
Suppose
$1<d<\infty$, $1<p_i<\frac{2d}{d+1}$,
 $p_i\le q_i\le 2$, for $i=0,1$, moreover
$p_0\neq p_1$, $p_0^{-1}-q_0^{-1}=p_1^{-1}-q_1^{-1}$.
 Then
\Be\label{interpolreal} [\fM^{p_0,q_0,\sigma_0}_{d}, \fM^{p_1,q_1,\sigma_1}_{d}]_{\vth, \infty} =
\fM^{p,q,\infty}_{d},
\Ee
for $(1/p,1/q)=(1-\vth)(1/p_0,1/q_0)+\vth(1/p_1,1/q_1)$ with
 $0<\vth<1$.
\end{corollary}

\begin{proof}[{Proof of Corollary \ref{interpolcorreal}}]

We first observe that for a compatible pair of Banach spaces $A_0, A_1$ we have the formula
\begin{equation}\label{vectinterpol}
[\ell^\infty_b (A_0), \ell^\infty_b (A_1)]_{\vth, \infty}
= \ell^\infty_b ([A_0,A_1]_{\vth,\infty})
\end{equation}
This follows quickly from the definition of the $K_{\vth,\infty}$ method
(and interchanging two suprema).

We now set $w(r)=(1+|r|)^{-(d-1)/2}$, $d\nu(r)=(1+|r|)^{d-1}$, and let
$L^{q,\sigma}(w, d\nu)$
be the space of functions $f$ for which
$fw$ belongs to Lorentz space $L^{q,\sigma}(d\nu)$
(and the norm is given by $\|fw\|_{L^{q,\sigma}(d\nu)}  $ where we work with a suitable norm on the Lorentz space).
The standard interpolation formulas for Lorentz spaces apply
and by \eqref{vectinterpol} we have
for $q_0\neq q_1$ and $1/q=(1-\vth)/q_0+\vth/q_1$,
\begin{equation*}
[\ell^\infty_b (L^{q_0,\sigma_0}(w, d\nu)),
\ell^\infty_b (L^{q_1,\sigma_1}(w, d\nu))]_{\vth, \infty}
= \ell^\infty_b (L^{q,\infty}(w, d\nu)).
\end{equation*}
Now let
$\floc^{q,\sigma}_b(w,d\nu)$ be the space of all $m$
which are
integrable  over every compact subinterval of $(0,\infty)$
and satisfy the condition
\[
\sup_{t>0} t^b  \big\|\cF_\bbR^{-1}[\phi m(t\cdot)]
\big\|_{L^{q,\sigma}(w, d\nu)} <\infty.
\]
Then the arguments in the proof of Lemma \ref{interpolfloc} show that
the maps $\fA$, $\fB$ defined in \eqref{Adef}, \eqref{Bdef}
can be used to
show that
$\floc^{q,\sigma}_b(w,d\nu)$ is a retract of
$\ell^\infty_b (L^{q,\sigma}(w, d\nu))$.
One deduces quickly that for $q_0\neq q_1$
\begin{equation*}
[\floc^{q_0,\sigma_0}_b(w,d\nu),
\floc^{q_1,\sigma_1}_b(w,d\nu)]_{\vth,\infty}=
\floc^{q,\infty}_b(w,d\nu)
\end{equation*}
and the asserted result follows from Theorem \ref{mainhankel}
 if we apply the last formula to  the spaces  $\fM^{p,q,\sigma}_d$
with fixed $d$ and fixed
$b=d(1/p-1/q)$.
\end{proof}

\noi{\bf Remarks on compactly supported multipliers.} The proofs show that for multipliers which are compactly supported away from the origin the result of Theorem \ref{mainhankel} can be sharpened.

\begin{theorem}\label{mainhankelcomp}
Let $m$ be compactly supported and integrable in $(0,\infty)$.
Suppose
$1<d<\infty$, $1<p<\frac{2d}{d+1}$, $p\le q< 2$ and
$1\le \sigma\le \infty$.
Then the following statements are equivalent.

(i) $T_m$ maps $L^{p,\sigma}(\mu_d)$
boundedly to $L^{q,\sigma}(\mu_d)$.

(ii) $T_m$ maps $L^{p,1}(\mu_d)$
boundedly to $L^{q,\sigma}(\mu_d)$.

(iii)  $\|\cB_d[m]\|_{L^{q,\sigma}(\mu_d)}<\infty.$

(iv)
$\big\| (1+|\cdot|)^{-\frac{d-1}{2} } \cF^{-1}_\bbR[m]
\big\|_{L^{q,\sigma}((1+|x|)^{d-1} dx)} <\infty$.
\end{theorem}

A similar  statement can be formulated for the analogue of Theorem
\ref{main} (again for $m$ supported in $(1/2,2)$).
 In particular for the case $\sigma=\infty$ we see that then  the  restricted
weak type $(p,p)$ inequality, the weak type $(p,p)$ inequality and the
stronger $L^{p,\infty}_\rad\to L^{p,\infty}_\rad$ bound are all equivalent in the range
$1<p<\frac{2d}{d+1}$.
We note that for the case of Bochner-Riesz multipliers
such  endpoint $L^{p,\infty}_\rad$ bounds had been obtained by Colzani, Travaglini and Vignati \cite{ctv}, extending earlier weak type endpoint
bounds by Chanillo and Muckenhoupt \cite{chanmucken}. The result for Bochner-Riesz means follows from the above theorem
(after separately dealing with the irrelevant part of the multiplier near $0$).
 This phenomenon
has no analogue for Fourier multipliers on $\bbR^d$ since
$L^{p,\infty}\to L^{p,\infty}$ boundedness for translation invariant
operators on $\bbR^d$ already implies $L^p\to L^p$ boundedness
(\cite{colzani}, \cite{steinberg}).

The proof of Theorem \ref{mainhankelcomp} is essentially  the same
as the proof of Theorem \ref{mainhankel}, but more elementary
since only a finite number of dyadic scales on the multiplier side
are involved hence  no Littlewood-Paley theory and singular
integral estimates are needed. The  difference (and improvement)
in  condition (i), and the extended range of $\sigma$  come from
Proposition \ref{Hprop} which involves only one dyadic scale and
the space $L^{p,\infty} (\mu_d)$ on the right hand side of
\eqref{Hest}.

\section{Open problems}\label{open}

\subsection{\it Radial Fourier multipliers}
Let $K$ be a radial convolution kernel on $\bbR^d$, $d\ge 2$.

\noi{\it Question:} Is there a $p>1$ for which the condition
\eqref{LpscondonK} (with $\sigma=p$)
 implies that the convolution operator $f\mapsto K*f$ is bounded on $L^p(\bbR^d)$?

The local version of this is open as well:

\noi{\it Question:}
Suppose that $K$ is radial and $\widehat K$ is compactly supported in
$\bbR^d\setminus\{0\}$.
Is there a $p>1$ for which the condition $K\in L^p(\bbR^d)$
implies that the convolution operator $f\mapsto K*f$ is bounded on
$L^p(\bbR^d)$?

It is known (\cf. \cite{ms-adv}) that under a  slightly weaker condition
than \eqref{LpscondonK},
namely the finiteness of $\sup_{t>0}\|\Phi* K_t\|_{L^p((1+|x|)^\eps)}$
for some $\eps>0$ implies $L^p$ boundedness for certain $p>1$. The condition on $p$ is that for
the dual exponent $p'$ the  local smoothing problem for
the wave equation in $\Bbb R^{d+1}$ can be solved up to endpoint estimates.
Wolff \cite{wolff} proved
 such estimates for $d=2$ and large $p'$; for corresponding results in higher dimensions see \cite{LW},
and for the currently known ranges
of Wolff's inequality  see \cite{gss}.

It is likely that in order to prove or come closer to a characterization
one needs to prove an endpoint version of Wolff's inequality.
The currently known method of proof (by induction on scales)
fails to give such sharp bounds.

\subsection{\it Localized Besov conditions.}
Short of a characterization one can ask whether for some $p>1$
the $L^p$ condition of Corollary \ref{besovcor}
\[\sup_{t>0} \|\varphi m(t\cdot)\|_{B^2_{d(\frac 1p-\frac 12),p}}<\infty
\]
implies that $m(|\cdot|)$ is a multiplier of $\cF L^p(\bbR^d)$. Again
the analogous question for $m$ supported in $(1/2,2)$ is also open.
 A result which comes close is in \cite{se-ind}.
There a scale of spaces $R^p_{\alpha,s}$
is introduced with $B^p_{\alpha,1} \subset R^p_{\alpha,s}\subset
B^p_{\alpha,p}$ for $1\le s\le p$ and $L^p(\bbR^d)$ boundedness is proved
under the condition
$\sup_{t>0} \|\varphi m(t\cdot)\|_{R^2_{d(1/p-1/2),p}}<\infty$,
for $1<p\le\frac{2(d+1)}{d+3}$.

\subsection{\it Localized multiplier conditions} Does the analogue of
Corollary \ref{gl-loccorr} hold for radial Fourier multipliers, acting on general functions in $L^p(\bbR^d)$,  some $p>1$?

\subsection{\it Hankel multipliers in the complementary range.}
No nontrivial characterization just
in terms of the convolution kernel seems to be  known
(and perhaps  may not be expected) for the range $\frac{2d}{d+1}\le p<2$.

\medskip

{\it Addendum, January 2008.} Very recently, after the submission of this paper, F. Nazarov and the second author made some progress concerning the 
problems  on radial Fourier multipliers.
The manuscript \cite{nase} contains characterizations for given 
$p$, provided  that the dimension $d$ is large enough. Presently,
there are no optimal results concerning the range of $p$;  moreover, in 
dimensions $2,3,4$, the problems are still open for any $p\in (1,2)$.

\appendix
\section{On Zafran's result}
\label{appendix}
Recall that for a compatible couple of Banach spaces $(A_0,A_1)$ a
space
$X\subset A_0+A_1$ is called an interpolation space for $(A_0,A_1)$ if
there is a  constant $C$ so that for every $T:A_0+A_1\to A_0+A_1$ which is bounded on $A_0$ and bounded on $A_1$ we have
\Be \|T\|_{X\to X} \le C \max \big\{
\|T\|_{A_0\to A_0}, \|T\|_{A_1\to A_1}\big\}.
\Ee
Zafran \cite{z} showed that the space $M^p(\bbR)$ is not an interpolation space for the pair $M^1(\bbR)$ (the Fourier transforms of bounded Borel measures)
and $M^2(\bbR)=L^\infty(\bbR)$. His arguments in conjunction
 with Bourgain's theorem on $\Lambda(p)$ sets
can be extended to show

\begin{proposition}
\label{appprop}
Let $1\le p_0< p< p_1\le 2$. Then $M^p(\bbR)$ is not an
interpolation space between $M^{p_0}(\bbR)$ and $M^{p_1}(\bbR)$.
\end{proposition}

We start quoting a standard result on
random Fourier series due to Salem and Zygmund (\cite{salz}, ch. IV);
it can be
 proved from the distribution inequality for Rademacher expansions and Bernstein's inequality for trigonometric polynomials.
Let $r_k$ be the sequence of Rademacher functions and
define $$F_R(t,\theta)=\sum_{k=1}^R a_k e^{ik\theta} r_k(t).$$
Then there is a constant $C$ so that for all  integers $R\ge 2$,
and for  $1\le \rho<\infty$
\Be \label{explinfty}
\Big(\int_0^1   \sup_\theta|F_R(t,\theta)|^\rho dt\Big)^{1/\rho} \le C
 \sqrt {\rho\log R}
(\sum_k|a_k|^2)^{1/2}.
\Ee
By the standard averaging argument the $\log R$ term may be dropped if the supremum in $\theta$ is replaced by an
$L^\rho$ norm.


The proof of Proposition \ref{appprop} relies on a deep result
by Bourgain \cite{bourgain}
(proved earlier
by Rudin \cite{rudin} for  $p'$  an even integer).

\noi{\bf Bourgain's theorem}. {\it Let $1<p\le 2$, $p'=p/(p-1)$. There is a
constant $C_p$ with the following property. For each
integer $N\ge 2$ there exists a
set $S_N$ of cardinality $N$ which consists of integers in $[0,N^{p'/2}]$ so that}
\Be \label{Lambda}
\Big(\int_0^{2\pi} \Big|\sum_{k\in S_N}a_k e^{ikx} \Big|^{p'}
dx\Big)^{1/p'} \le
C_p \Big(\sum_{k}|a_k|^2\Big)^{1/2}.
\Ee
In what follows we shall always fix $p$ and the associated family of
 sets $S_N$ for which \eqref{Lambda} holds.
A consequence of \eqref{Lambda}  is that
\Be \label{bourgmult} \Big\| \sum_{k\in S_N} b_k \eta(\cdot-k)
\Big\|_{M^p(\bbR)} \le C(p) \sup_k|b_k|
\Ee
where $\eta$ is the Fej\'er multiplier $$\eta(\xi)=\begin{cases}
1-|\xi|, \quad &|\xi|\le 1,
\\
0, \quad &|\xi|>1.
\end{cases}
$$
To see \eqref{bourgmult} we first note that  \eqref{Lambda} implies that 
the sequence
$\{b_k \chi_{S_N}(k)\}_{k\in \bbZ}$ defines  a multiplier in
$M_2^{p'}(\bbZ)$, and by duality a multiplier in $M_p^2(\bbZ)$, with norms bounded by $C\|b\|_{\ell^\infty}$ 
(by this we mean that the corresponding convolution operator maps $L^p(\bbT)$ to $L^2(\bbT)$ with norm $\lc
\|b\|_{\ell^\infty}$). By H\"older's inequality (using the compactness of $\bbT$) it also follows that this sequence belongs to $M_p(\bbZ)$, since
$p\le2$.
Now Jodeit's extension result  \cite{jodeit}
(see also \cite{ftg})  for multipliers in $M^p(\bbZ)$ says
\Be\label{ext}
\big\|\sum_k m_k \eta(\cdot-k)\big\|_{M^p(\bbR)}\lc
 \big\|\{m_k\}\big\|_{M^p(\bbZ)}, \quad 1\le p\le q\le \infty.
\Ee
Inequality \eqref{bourgmult} follows.

A third ingredient will be a  sequence of multipliers $h_N$ which belong to all $M^q(\bbR)$ classes and satisfy the lower and upper bounds
\begin{equation}\label{b}
    \|h_N\|_{M^q(\bbR)}\approx N^{\frac1q-\frac12},\quad
\text{for } 1\le q\le 2.
\end{equation}
There are many examples of such families, we choose
$$h_N(\xi)= \chi(\xi) e^{iN|\xi|^2}$$
where
$\chi$ is a smooth function supported in $(1/2,2)$ which is equal to one on
$[3/4,5/4]$. To see that \eqref{b} holds true we examine the kernel
$K_N=\cF^{-1}[h_N]$.
By stationary phase arguments we see that
$|K_N(x)|\lc N^{-1/2}$ for $N/4\le x\le 4N$ and
$|K_N(x)|\ge c N^{-1/2}$ for $3N/2\le x\le 5N/2$; moreover by
integration by parts  $|K_N(x)|\le C_L x^{-L}$ for $x\ge 4N$ and
$|K_N(x)|\le C_L N^{-L}$ for $x<N/4$. This shows that
$\|K_N\|_{L^q}\ge c N^{1/q-1/2}$ and since $h_N$ has compact support
this implies the lower bound in \eqref{b}. The kernel calculation also
implies the upper bound for $q=1$ and interpolation
with the trivial $L^2$ bound yields \eqref{b}.

\begin{proof} [Proof of Proposition \ref{appprop}.]
Let $N$ be a large integer, $R\gg N$
 and let $S_N$ be a set in $[0,R]$ so that \eqref{bourgmult}
holds 
(by Bourgain's theorem we may choose $R\approx N^{p'/2}$).
Essentially following Zafran we then consider the rank one  operators
$L_N:M^q\to M^q$ defined by
\[
L_N(m)= v_N(m) h_N,  \quad\text{ where } v_N(m)=\frac{1}{N}
\sum_{k\in S_N} \rho_k \int m(\xi)  \eta(\xi-k) d\xi.
\]
Here we assume that $\rho_k\in \{1,-1\}$
are chosen so that
\begin{equation}\label{a}
    \sup_{x} \Big| N^{-1} \sum_{k\in S_N} \rho_k e^{ikx}\Big| \le C N^{-1/2}
\sqrt{\log R};
\end{equation}
this can be achieved
by \eqref{explinfty}.

We shall show that
\Be \label{upperbound} \|L_N\|_{M^q\to M^q} \le C
\min\{ N^{1/q-1/2}, N^{-1+1/q} R^{1/q'}\sqrt{\log R} \}, \quad 1\le q\le 2;
\Ee
moreover if $p$ is as in \eqref{bourgmult} then
\Be \label{lowerbound} \|L_N\|_{M^p\to M^p} \gtrsim\,N^{1/p-1/2}.
\Ee

We first show that the validity of \eqref{upperbound} and \eqref{lowerbound} implies the assertion of the Proposition.
Namely  if $M^p(\bbR)$ were  an interpolation space
of $(M^{p_0}(\bbR), M^{p_1}(\bbR))$, with $p_0<p<p_1$,  then
\Be\label{interpolspace}
\|L_N\|_{M^p\to M^p}\leq\,\cC\max\,\{\|L_N\|_{M^{p_0}\to
M^{p_0}},\|L_N\|_{M^{p_1}\to M^{p_1}}\}.
\Ee
We use the first bound  in
\eqref{upperbound} for $q=p_1$ and the second one for $q=p_0$. Thus by
\eqref{interpolspace} and \eqref{lowerbound}
\[ N^{1/p-1/2}\lc \cC \max\{ R^{1/p_0'}\sqrt{\log R} N^{-1/p_0'}, \,
N^{1/p_1-1/2} \}.\]
By Bourgain's theorem we may choose $N$ large and
$R\approx N^{p'/2}$.
Since $p<p_1$, the last displayed inequality implies
$1/p-1/2\le (\frac{p'}2-1)/p_0' $
which solving for $p$ is equivalent to $p\leq p_0$, a
contradiction.

\noi{\it Proof of \eqref{upperbound}}.
We set  $\omega_N:= N^{-1}\sum_{k\in S_N} \rho_k \cF^{-1}[ \eta(\cdot-k)]$.
Since the Fej\'er kernel $\cF^{-1}[\eta]$ belongs to $L^1\cap L^\infty$
we observe that  
$\|\cF^{-1}[\eta]\|_{L^r}<\infty$
and hence by \eqref{a}
\Be \label{LqbdomegaN}
\|\om_N\|_{L^r(\bbR)} \lc N^{-1/2} \sqrt{\log R}, \quad 1\le r< \infty.
\Ee

In view of \eqref{b} the inequality \eqref{upperbound} follows from
\Be\label{linfctl}
|v_N(m)|\lc \min\{ 1, N^{-1/2} R^{1/q'}\sqrt{\log R}
\} \,\|m\|_{M^q}, \quad 1\le q\le 2.
\Ee
The first bound in
\eqref{linfctl}
is obvious since
$|v_N(m)|\le \|m\|_\infty$. The second
follows from Plancherel's theorem. To see this
let
$\zeta_R(\xi)= \zeta_0(\xi/R)$ where $\zeta_0$ is compactly supported with the property that $\zeta_0(\xi)=1$ for $|\xi|\le 2$. Then by
\eqref{LqbdomegaN}
\begin{align*}
|v_N(m)|&= c\Big|\int \cF^{-1}[m \zeta_R] (x)
\omega_N(x) dx\Big|
\\ &\le \|\omega_N\|_{q'} \|\cF^{-1}[m \zeta_R]\|_q
\le \|\omega_N\|_{q'}
\|m\|_{M^q}
\|\cF^{-1}[ \zeta_R]\|_q
\end{align*}
and  the second bound in  \eqref{linfctl} follows if we observe that
$\|\cF^{-1}[ \zeta_R]\|_q=O(R^{1/q'})$. Thus  \eqref{upperbound} is proved.

\noi{\it Proof of \eqref{lowerbound}.}
Here we use  \eqref{bourgmult} (which  was a consequence
of the crucial
$\Lambda(p')$ estimate for the set $S_N$).
We apply $L_N$ to $\widehat \omega_N$ and obtain
\[
\|L_N\|_{M^p\to M^p}\geq  \,\frac{\|L_N(\widehat\omega_N)\|_{M^p}}
{\|\widehat \omega_N\|_{M^p}}
=
\frac{ \|\widehat \omega_N\|_2^2 \|h_N\|_{M^p}}
{\|\widehat \omega_N\|_{M^p}}
\gtrsim\,\|h_N\|_{M^p}
\]
where we have used that $\|\widehat \omega_N\|_2^2 \approx N^{-1}$ and
$N\|\widehat \omega_N\|_{M^p} \lc 1$,
by
\eqref{bourgmult}. Thus
\eqref{lowerbound}  follows from \eqref{b}.
\end{proof}

\end{document}